\documentclass[reqno]{amsart}
\usepackage{amsmath,amssymb}
\usepackage{amsmath,amssymb,cite}
\usepackage{graphics,epsfig,color}
\usepackage[mathscr]{euscript}
\usepackage{xcolor}

\theoremstyle{plain}
\newtheorem{theorem}{Theorem}[section]
\newtheorem{proposition}[theorem]{Proposition}
\newtheorem{lemma}[theorem]{Lemma}
\newtheorem{corollary}[theorem]{Corollary}

\newtheorem{definition}[theorem]{Definition}

\newtheorem{remark}[theorem]{Remark}

\numberwithin{equation}{section}
\numberwithin{theorem}{section}

\newcommand{\mc}[1]{{\mathcal #1}}
\newcommand{\mb}[1]{{\mathbf #1}}
\newcommand{\mf}[1]{{\mathfrak #1}}
\newcommand{\bs}[1]{{\boldsymbol #1}}
\newcommand{\bb}[1]{{\mathbb #1}}
\newcommand{\ms}[1]{{\mathscr #1}}

\newcommand{\<}{\langle}
\renewcommand{\>}{\rangle}
\renewcommand{\Cap}{{\rm cap}}

\title[$\Gamma$-expansion of large deviations rate functional]
{Full $\Gamma$-expansion of reversible Markov chains level two
large deviations rate functionals}

\author{Claudio Landim} \address{IMPA, Estrada Dona Castorina 110,
J. Botanico, 22460 Rio de Janeiro, Brazil and Univ. Rouen Normandie,
CNRS, LMRS UMR 6085, F-76000 Rouen, France.}  \email{landim@impa.br}

\author{Ricardo Misturini} \address{ UFRGS, Instituto de Matem\'atica,
Campus do Vale, Av. Bento Gon\c calves, 9500. CEP 91509-900, Porto
Alegre, Brasil} \email{ricardo.misturini@ufrgs.br}

\author{Federico Sau} \address{Università degli Studi di Trieste,
Dipartimento di Matematica e Geoscienze, Via Valerio 12/1, 34127
Trieste, Italy} \email{federico.sau@units.it}

\begin{document}

\begin{abstract}
Let $\Xi_n \subset \bb R^d$, $n\ge 1$, be a sequence of finite sets
and consider a $\Xi_n$-valued, irreducible, reversible,
continuous-time Markov chain $(X^{(n)}_t:t\ge 0)$.  Denote by
$\ms P(\bb R^d) $ the set of probability measures on $\bb R^d$ and by
$\ms I_n\colon \ms P(\bb R^d) \to [0,+\infty)$ the level two large
deviations rate functional for $X^{(n)}_t$ as $t\to\infty$. We present
a general method, based on tools used to prove the metastable
behaviour of Markov chains, to derive a full expansion of $\ms I_n$
expressing it as
$\ms I_n = \ms I^{(0)} \,+\, \sum_{1\le p\le \mf q} (1/\theta^{(p)}_n)
\, \ms I^{(p)}$, where
$\ms I^{(p)}\colon \ms P(\bb R^d) \to [0,+\infty]$ represent rate
functionals independent of $n$ and $\theta^{(p)}_n$ sequences such
that $\theta^{(1)}_n \to\infty$,
$\theta^{(p)}_n / \theta^{(p+1)}_n \to 0$ for $1\le p< \mf q$.  The
speed $\theta^{(p)}_n$ corresponds to the time-scale at which the
Markov chains $X^{(n)}_t$ exhibits a metastable behaviour, and the
$\ms I^{(p-1)}$ zero-level sets to the metastable states. To
illustrate the theory we apply the method to random walks in potential
fields.
\end{abstract}

\noindent
\keywords{Large deviations, $\Gamma$-convergence, Metastability,
  Continuous-time Markov processes on discrete state spaces}

\subjclass[2020]
{Primary 60F10; 60J27; 60J45, 60K35}

\maketitle
\thispagestyle{empty}

\section{Introduction}
\label{sec0}

With the aim of rigorously describing the metastable behaviours
emerging in all sorts of natural phenomena, the research on
metastability of Markov processes registered a considerable
development over the past decades. See, e.g., the books \cite{OV05,
BH15} and the review article \cite{lrev} and references therein
for recent accounts on the subject and the diverse approaches
developed. The common feature of all these mathematical theories rests
on the identification of two or more locally stable states, together
with a clear separation of scales: persistence in a state and
transitions from state to state may be observed over time spans of
different order.  Furthermore, when the system comprises a complex
energetic landscape, several metastable states and scales, naturally
ordered, arise (cf. \cite{bgl-lm} and references therein).

The goal of this article is that of developing an analytical framework
to describe the hierarchical metastability structure in the context of
continuous-time Markov chains $X_t^{(n)}$ on a finite state space
$\Xi_n$ of growing size, satisfying natural assumptions. In the spirit
of the recent works \cite{GM17,bgl-lm,l-mld}, we express the
metastable structure of $X_t^{(n)}$ through an expansion, in the limit
$n\to \infty$, of the corresponding level two large deviations rate
functional $\ms I_n$ \cite{dv75}, which describes the deviations of
the empirical measure of $X_t^{(n)}$ as $t\to \infty$ (see
\eqref{48}--\eqref{50}).  This expansion is informally represented as
\begin{equation}
	\label{f05}
	\ms I_n  \;=\; \ms I^{(0)} \;+\; \sum_{p=1}^{\mf q}
	\frac{1}{\theta^{(p)}_n} \, \ms I^{(p)}\;,
\end{equation}
where $\theta_n^{(p)}\in (0,\infty)$ are diverging constants referred to as \emph{speeds}, and $\ms I^{(p)}$ are functionals. 
 Since the topology employed to perform \eqref{f05} is that induced by the notion of $\Gamma$-convergence, we refer to \eqref{f05} as a \emph{full $\Gamma$-expansion of the functional $\ms I_n$} (Definition \ref{def1}).
 
 Our main results, Theorems \ref{mt1} and \ref{mt2}, 	establish that the speeds $\theta^{(p)}_n$ appearing in \eqref{f05} correspond to the time-scales at which the chain exhibits a
metastable behaviour. Furthermore, the functionals $\ms I^{(p)}$
coincide with the level two large deviations rate functionals of the
reduced Markov chain, which describes the evolution of the process
$X^{(n)}_t$ among the wells in the time-scale $\theta^{(p)}_n$. In
particular, the $0$-level sets of the functionals $\ms I^{(p)}$ is
formed by convex combinations of the metastable states of the Markov
chain in the time-scale $\theta^{(p+1)}_n$.

 Our original motivating example, discussed in detail in Section \ref{sec6}, is that of  random walks in potential fields evolving on  the  discrete $d$-dimensional torus of mesh size $1/n$,  whose metastable
 behaviour has already been considered in \cite{lmt2015, ls2018}. However,  instead of proving directly the full $\Gamma$-expansion  for this specific example, we develop a more general method primarily based on two assumptions, namely, conditions {\rm (H0)} and {\rm (H1)} in Section \ref{sec:Hs} below. Roughly speaking, {\rm (H0)} is what allows to start a metastability analysis, by defining a reduced model with a finite number of states, whereas {\rm (H1)}  on the convergence of mean rates (see \eqref{20}) for the trace process  encodes the  Markov behaviour of the reduced model. These two conditions lie at the core of the so-called martingale approach to metastability and have been verified for several metastable dynamics of interest, including zero-range processes and spin systems at low temperature; for more details, we refer to Remark \ref{r-2}.
The other conditions {\rm (H2) -- (H5)} which we assume for our first main result Theorem \ref{mt1} are either structural/natural (e.g., {\rm (H2)} on the nested structure of the support of the metastable states, or {\rm (H3)} on the finiteness of the number of scales) or technical (e.g., {\rm (H4b)} on the size of the wells, or {\rm (H5)} on fast relaxation within the well), to be checked in each example. Theorem \ref{mt2} replaces part of  {\rm (H4)} with condition {\rm (H6)}, more natural in the setting of condensing particle systems.

The key analytical tool in our analysis is the $\Gamma$-convergence of rate functionals. More precisely, since we consider the underlying state space $\Xi_n$ to be a discrete approximation of a limiting base space $\Xi$, we state our results in terms of $\Gamma$-convergence of non-negative functionals $\ms I_n$ on the space $\ms P(\Xi)$ of probability measures on $\Xi$.
For precise definitions and more details on the role of $\Gamma$-convergence in large deviations theory, we refer to the next section. Let us here  emphasize that condition {\rm (H5)}, which is not present in \cite{bgl-lm,l-mld}, is crucial in our proof of the \textit{$\Gamma$-liminf} inequality to infer some quantitative features for weakly converging sequences with finite energy.

As already mentioned, our results relate to those obtained in \cite{GM17,bgl-lm,l-mld}.	
More specifically, the expansion in \eqref{f05} has been proved in \cite{GM17} for overdamped Langevin
dynamics in which the potential has a finite number of local minima
and each local minimum is separated from the others by a unique
optimal saddle point. In \cite{bgl-lm, l-mld}, this expansion has been derived for finite
state Markov chains with a fixed state space  using tools developed in \cite{bl2,
	bl7} to prove the metastable behaviour of Markov chains. 

We emphasize that, while our example of random walks in potential fields may be thought of as the discrete analogue of the model considered in \cite{GM17}, our assumptions and strategy differ considerably from those in \cite{GM17}. Indeed, \cite{GM17}, which is based on semi-classical analysis techniques, strongly relies on the  non-overlap of low-lying eigenvalues of the associated generator, a condition ensured by the assumption of having at most one  minimum on each level. Our analysis based on capacity estimates does not prevent us to consider a setting with multiple wells at the same height. 

Moreover,  compared to the works \cite{bgl-lm,l-mld} concerned with Markov chains on a fixed state space $\Xi_n= \Xi$, our setting of growing state space covers a much broader list of examples. As a drawback, our analysis and assumptions are  more involved, compensating for the substantial simplification that, in case of $\Xi$ finite, all topologies on $\ms P(\Xi)$ are equivalent. For instance,  \cite{bgl-lm} merely requires the jump rates to satisfy a ratio assumption, a request which may be thought of as analogous to our condition {\rm (H1)} (see \cite[Eqs.\ (2.4) \& (2.10)]{bgl-lm}). In particular, \cite{bgl-lm} requires none of our technical conditions.

Before presenting in full detail our setting, assumptions and main results in the next section, we conclude this introduction by mentioning that, whereas we work in a reversible setting, we expect our results to carry over  to the non-reversible setting as well, via an approach similar to that in \cite{l-mld} and with an analogous set of assumptions. Along the same lines, we believe a similar analysis based on capacities  to yield a full $\Gamma$-expansion for non-reversible diffusions in potential fields, extending the results in \cite{GM17} obtained in the reversible context.

\section{Setting and main results}
\label{sec1}
\subsection{Setting}\label{sec:setting}
Let $\Xi$ be the closure of an open and bounded subset of $\bb R^d$,
or the $d$-dimensional torus.  Denote by $\Xi_n$ the discretization of
$\Xi$: $\color{blue} \Xi_n = \Xi \cap (n^{-1} \bb Z^d)$, $n\ge 1$,
where $n^{-1} \bb Z^d = \{k/n : k\in \bb Z^d\}$. The elements of
$\Xi_n$ are represented by the symbols $x$, $y$ and $z$.

Let $\color{blue} X^{(n)}_t$, $t\ge 0$, be a $\Xi_n$-valued,
reversible, irreducible, continuous-time Markov chain.  Denote the
jump rates by
$\color{blue} R_n\colon \Xi_n \times \Xi_n \to \bb R_+:=[0,\infty)$,
and the generator by $\ms L_n$, so that
\begin{equation}
\label{47}
{ \color{blue} (\ms L_n f)(x)}
\;=\; \sum_{y \in \Xi_n} R_n(x,y)\, \{\,
f(y)\,-\, f(x)\,\}\;, \quad f\colon \Xi_n\to \bb R\;.
\end{equation}
The holding times are represented by $\color{blue}\lambda_n(x)$ and
the jump probabilities by $\color{blue}p_n(x,y)$ so that
$R_n(x,y) = \lambda_n(x) \, p_n(x,y)$.  Let $\color{blue} \pi_n$ be the
unique stationary state, assumed to satisfy the detailed-balance
conditions.

\subsection*{Large deviations}

Denote by $\color{blue} \ms P(\Omega)$, $\Omega\subset\bb R^d$, the
set of probability measures on $\Omega$ endowed with the weak
topology, and by $L^{(n)}_t$ the empirical measure of the chain
$X^{(n)}_t$ defined as:
\begin{equation}
\label{48}
{ \color{blue} L^{(n)}_t}
\;:=\; \frac{1}{t} \int_0^t \delta_{X^{(n)}_s}\; ds \;,
\end{equation}
where $\color{blue} \delta_x$, $x\in \Xi$, represents the Dirac
measure concentrated at $x$. Thus, $L^{(n)}_t$ is a random element of
$\ms P(\Xi_n)$ and $L^{(n)}_t (\ms A)$, $\ms A \subset \Xi_n$, stands for the
proportion of time the process $X^{(n)}_t$ stays at $\ms A$ in the time
interval $[0,t]$.

As the Markov chain $X^{(n)}_t$ is irreducible, by the ergodic
theorem, for any starting point $x\in \Xi_n$, as $t\to\infty$, the
empirical measure $L^{(n)}_t$ converges in probability to the
stationary state $\color{blue} \pi_n$.

Donsker and Varadhan \cite{dv75} proved the associated large
deviations principle. More precisely, they showed that for any subset
$A$ of $\ms P(\Xi_n)$,
\begin{equation}
\label{50}
\begin{aligned}
-\, \inf_{\mu\in A^o} \ms I_n (\mu) \;&\le\; \liminf_{t\to \infty}
\min_{x\in \Xi_n}\, \frac{1}{t}\,
\log \mb P^{(n)}_x \big[ \, L^{(n)}_t \in A\,\big] \\
&\qquad \;\le\; \limsup_{t \to \infty} \max_{x\in \Xi_n}\,
\frac{1}{t}\, \log \mb P^{(n)}_x \big[ \, L^{(n)}_t \in A\,\big]
\;\le\; -\, \inf_{\mu\in \overline{A}} \ms I_n (\mu)\;.
\end{aligned}
\end{equation}
In this formula, $\color{blue} \mb P^{(n)}_{\! x}$,
$x\in \Xi_n$, stands for the distribution of the process $X^{(n)}_t$
starting from $x$, $A^o$, $\overline{A}$ represent the
interior, closure of $A$, respectively, and
$\ms I_n \colon \ms P(\Xi_n) \to [0,+\infty)$ is the level two large
deviations rate functional given by
\begin{equation}
\label{f4}
{\color{blue} \ms I_n (\mu)} \;:=\; \sup_u \,-\,
\sum_{x\in \Xi_n} \frac{(\ms L_n u)(x)}{u(x)} \; \mu(x)\;,
\end{equation}
where the supremum is performed over all functions
$u: \Xi_n \to (0,\infty)$.  Since we assumed reversibility and
$\pi_n(x)>0$ for all $x\in \Xi_n$, by \cite[Theorem 5]{dv75},
\begin{equation}
\label{f6}
\ms I_n(\mu)
\;=\; \< \,
\sqrt{f_n} \,,\, (-\, \ms L_n) \sqrt{f_n} \,\>_{\pi_n}\;, 
\end{equation}
for all measures $\mu \in \ms P(\Xi_n)$, where
$f_n(x) = \mu(x)/\pi_n(x)$.

\subsection*{$\Gamma$-convergence}

We investigate in this article the $\Gamma$-convergence of the action
functional $\ms I_n$. We first recall its definition.

Fix a Polish space $\mc X$ and a sequence $(U_n : n\in\bb N)$ of
functionals on $\mc X$, $U_n\colon \mc X \to [0,+\infty]$.  The
sequence $U_n$ \emph{$\Gamma$-converges} to the functional
$U\colon \mc X\to [0,+\infty]$ if and only if the two following
conditions are met:

\begin{itemize}
\item [(i)]\emph{$\Gamma$-liminf.} The functional $U$ is a
$\Gamma$-liminf for the sequence $U_n$: For each $x\in\mc X$ and each
sequence $x_n\to x$, we have that
\begin{equation*}
\liminf_{n\to \infty} U_n(x_n) \;\ge\; U(x) \;.
\end{equation*}

\item [(ii)]\emph{$\Gamma$-limsup.} The functional $U$ is a
$\Gamma$-limsup for the sequence $U_n$: For each $x\in\mc X$ there
exists a sequence $x_n\to x$ such that
\begin{equation}
\label{30}
\limsup_{n\to \infty} U_n(x_n) \;\le\; U(x)\;.
\end{equation}
\end{itemize}

The role of the $\Gamma$-convergence of rate functionals in large
	deviations theory has long been established \cite{mar}. In our context, $\Gamma$-convergence of the rate functional $\ms I_n$ to 
$\ms I^{(0)}$ implies that, for every closed subset $F$ and open subset $G$ of
$\ms P(\Xi)$,
\begin{align*}
\begin{aligned}
\limsup_{n\to \infty} \, \limsup_{t\to \infty}
\, \frac{1}{t} \,\max_{x\in \Xi_n} \, \log\, 
\mb P_{\! x}^{(n)} \Big[\, L_t^{(n)}
\,\in\, F \, \Big] \;&\le\; -\, \inf_{\mu\in F} \ms I^{(0)} (\mu) \;,
\\
\liminf_{n\to \infty} \, \liminf_{t\to \infty}
\, \frac{1}{t}  \, \min_{x\in \Xi_n} \,\log\, 
\mb P^{(n)}_{\! x} \Big[\, L_t^{(n)}
\,\in\, G \, \Big] \;&\ge\; -\, \inf_{\mu\in G} \ms I^{(0)} (\mu) \;,
\end{aligned}
\end{align*}
where we recall that $L_t^{(n)}$ are the empirical measures defined in \eqref{48}.

$\Gamma$-convergence was a key tool in a number of problems involving large deviations principles. Among these, \cite{m10} deals with a 
large deviations principle for conservation laws as the viscosity and
the noise vanish simultaneously. Here, since the large deviations rate
functional vanishes at all weak solutions of the conservation law, the large deviations
have to be examined at a smaller speed in order to
distinguish entropic solutions from non-entropic ones. This last step requires tools from the
theory of $\Gamma$-convergence.

A related problem involving a double-limit procedure allows to recover  the large
deviations rate functional of one-dimensional asymmetric exclusion
processes, previously obtained in \cite{j00, v04, vil08, qt21}. In this setting, one first derives a large deviations principle for
one-dimensional weakly asymmetric exclusion processes, and then prove
the $\Gamma$-convergence of the rate functional as the diffusivity
vanishes (or as the asymmetry diverges).   In a slightly
different context, the $\Gamma$-convergence of the rate functional for
a fast-slow dynamics of $N$ interacting diffusions is proven in
\cite{bbc19}.

The $\Gamma$-convergence of the large deviations rate functional
$\ms I_n$ introduced in \eqref{f4} has been examined recently in
\cite{bgl, bd22} in the context of interacting particles systems to
show that in the large deviations principle one can interchange the
limits $n\to \infty$ and $t\to\infty$, where $1/n$ represents the
interdistance between particles. A similar result has been obtained in
\cite{bgl3} for diffusions in a potential field.

\subsection*{$\Gamma$-expansion}

In view of the previous results, assume that the functional $\ms I_n$
$\Gamma$-converges to a functional denoted by $\ms I^{(0)}$. If the
$0$-level set of $\ms I^{(0)}$ is a singleton there is nothing to
add. Otherwise, it is natural to consider the $\Gamma$-convergence of
$\theta^{(1)}_n \, \ms I_n$, for some sequence
$\theta^{(1)}_n \to +\infty$.

By \eqref{f4}, the rate functional $\theta^{(1)}_n \, \ms I_n$
corresponds to the level two large deviations rate functional of the
Markov chain induced by the generator $\theta^{(1)}_n \, \ms L_n$,
that is, to the Markov chain $X^{(n)}_t$ speeded-up by
$\theta^{(1)}_n$.

Assume that there exists a sequence $\theta^{(1)}_n \to +\infty$ for
which $\theta^{(1)}_n \, \ms I_n$ $\Gamma$-converges to a functional
denoted by $\ms I^{(1)}$. Clearly, $\ms I^{(1)} (\mu)$ is finite only
if $\mu$ belongs to the $0$-level set of $\ms I^{(0)}$.  We say that
we obtained the correct speed $\theta^{(1)}_n$ whenever
$\ms I^{(1)} (\mu)$ is finite if, and only if, $\mu$ belongs to the
$0$-level set of $\ms I^{(0)}$. In other words,
$\ms I^{(1)} (\mu) < \infty$ if, and only if, $\ms I^{(0)}
(\mu)=0$. If this is not the case, it means that there is an
intermediate scale $\theta'_n$, $\theta'_n \to\infty$,
$\theta'_n/\theta^{(1)}_n\to 0$, which has been missed.

If the $0$-level set of $\ms I^{(1)}$ is a singleton we end the
analysis of $\ms I_n$. Otherwise, we may iterate the procedure.  We
summarize these considerations in the next definition.

\begin{definition}
\label{def1}
Consider a functional $\ms I_n\colon \ms P(\Xi) \to [0,+\infty)$. A
\emph{full $\Gamma$-expansion of $\ms I_n$} is given by the speeds
$(\theta^{(p)}_n, n\ge 1)$, $1\le p\le \mf q$, and the functionals
$\ms I^{(p)}\colon \ms P(\Xi) \to [0, +\infty]$, $0\le p\le \mf q$, if:
\begin{itemize}
\item [(a)] The speeds $\theta^{(1)}_n, \dots, \theta^{(\mf q)}_n$ are
sequences such that $\theta^{(1)}_n \to\infty$,
$\theta^{(p)}_n / \theta^{(p+1)}_n \to 0$, $0\le p<\mf q$;

\item [(b)] $\ms I_n$ $\Gamma$-converges to $\ms I^{(0)}$, and for
each $1\le p\le \mf q$, $\theta^{(p)}_n \, \ms I_n$ $\Gamma$-converges
to $\ms I^{(p)}$;

\item [(c)] For $0\le p<\mf q$, $\ms I^{(p+1)} (\mu)$ is finite if,
and only if, $\mu$ belongs to the $0$-level set of $\ms I^{(p)}$;

\item [(d)] The $0$-level set of $\ms I^{(\mf q)}$ is a singleton.
\end{itemize}
\end{definition}

In the next section, we present a set of conditions on the Markov chain
$X^{(n)}_t$ which guarantee that the rate functional $\ms I_n$ can be
expanded as in \eqref{f05}, as defined in Definition \ref{def1}.
\subsection{Hypotheses}\label{sec:Hs}
We start extending the definition of $\ms
I_n$ to $\ms P(\Xi)$ by setting $\ms I_n(\mu) = + \infty$ for $\mu\not
\in \ms P(\Xi_n)$. Similarly, elements of $\ms P(\Xi_n)$ are
considered as measures on $\Xi$.

Recall the notation introduced in the previous section.  The starting
point of the analysis of the large deviations rate functional
$\ms I_n$ is the assumption that $\ms I_n$ $\Gamma$-converges to a
functional $\ms I^{(0)}\colon \ms P(\Xi) \to [0,+\infty]$ whose
$0$-level set consists of convex combinations of a finite number of
Dirac measures:

\begin{itemize}
\item [{\bf (H0)}] The sequence of functionals
$\ms I_n\colon \ms P(\Xi) \to [0,+\infty]$ $\Gamma$-converges to a
functional $\ms I^{(0)}\colon \ms P(\Xi) \to [0,+\infty]$. Moreover,
there exists a finite set
$\color{blue}\ms M = \{x_1, \dots , x_{\mf n_1}\}$ such that
$\{ \mu \in \ms P(\Xi) : \ms I^{(0)} (\mu)=0\} = \{\sum_{j\in S_1}
\omega_j\, \delta_{x_j} : \omega\in \ms P(S_1)\}$, where
$\color{blue} S_1 = \{1, \dots, \mf n_1\}$.
\end{itemize}

This hypothesis is analysed in Remark \ref{rm-1}.

\subsection*{Trace process}

Fix a non-empty subset $W$ of $\Xi_n$.  Denote by $T^{W}_n(t)$ the
total time the process $X^{(n)}_t$ spends in $W$ in the time-interval
$[0,t]$:
\begin{equation*}
{\color{blue} T^{W}_n(t) } \; :=\; \int_{0}^{t}\,\chi_{_{W}}(X^{(n)}_s)\; ds\;,
\end{equation*}
where $\color{blue}\chi_{_{W}}$ represents the indicator function of the
set $W$. Denote by $S^{W}_n(t)$ the generalized inverse of $T^{W}_n(t)$:
\begin{equation*}
{\color{blue}  S^{W}_n(t) } \;=\;\sup\{\,s\ge 0\,:\,T^{W}_n(s)\le t\,\}\;.
\end{equation*}

The trace of $X$ on $W$, denoted by
$\color{blue}(Y^{W}_t : t \ge 0)$, is defined by
\begin{equation}
\label{100}
Y^{W}_t\;=\; X^{(n)}_{S^{W}_n(t)} \;;\;\;\;t \ge 0\;.
\end{equation}
By Propositions 6.1 and 6.3 in \cite{bl2}, the trace process is an
irreducible, $W$-valued continuous-time Markov chain, obtained by
turning off the clock when the process $X^{(n)}_t$ visits the set
$W^{c}$, that is, by deleting all excursions to $W^{c}$. For this
reason, it is called the trace process of $X^{(n)}_t$ on $W$.

\subsection*{Metastable structure}

We next assume that the Markov chain $X^{(n)}_t$ exhibits a metastable
behaviour at different time-scales. The formulation of this condition
requires some notation.

Denote by $\color{blue} D(\bb R_+, A)$, $A$ a finite set, the space of
right-continuous functions $\mf x: \bb R_+ \to A$ with left-limits
endowed with the Skorohod topology and the associated Borel
$\sigma$-algebra. Let $\mb P^{(n)}_{\! x}$, $x\in \Xi_n$, be the
probability measure on the path space $D(\bb R_+, \Xi_n)$ induced by
the Markov chain $X^{(n)}_t$ starting from $x$. Expectation with
respect to $\mb P^{(n)}_{\! x}$ is represented by
$\color{blue} \mb E^{(n)}_x$.

Consider a partition
$\color{blue} \mf W = \{\ms W^{(1)}_n, \dots, \ms W^{(\mf m)}_n,
\Delta_n\} $ of $\Xi_n$, and let
$\color{blue} \ms W_n \,:=\, \cup_{j=1}^{\mf m} \ms W^{(j)}_n$.
Denote by $\color{blue} \{Y^{\ms W_n}_t: t\ge 0\}$ the trace of
$\{X^{(n)}_t: t\ge 0\} $ on $\ms W_n$.  By equation (2.5) in
\cite{lrev}, the jump rates of the trace process $Y^{\ms W_n}_t$,
represented by
$R^{\mf W}_n \colon \ms W_n \times \ms W_n \to \bb R_+$, are given by
\begin{equation}
\label{o-40}
{\color{blue}  R^{\mf W}_n (x,y)}
\;=\; \lambda_n(x) \; \mb P^{(n)}_{\! x} \big[ H_y
= H^+_{\ms W_n} \big]\;, \quad x\,,\; y \in \ms W_n \,, 
\; x\not = y  \;.
\end{equation}
In this formula, $\lambda_n(x)$ is the holding time of the chain
$X^{(n)}_t$ at $x$, and $H_{\ms A}$, $H^+_{\ms A}$,
${\ms A}\subset \Xi_n$, represent the hitting and return time of
${\ms A}$:
\begin{equation} 
\label{o-201}
{\color{blue} H_{\ms A} } \;: =\;
\inf \big \{t>0 : X^{(n)}_t \in {\ms A} \big\}\;,
\quad
{\color{blue} H^+_{\ms A}} \;: =\;
\inf \big \{t>\tau_1 : X^{(n)}_t \in {\ms A} \big\}\; ,  
\end{equation}
where $\tau_1$ stands for the time of the first jump of the chain
$X^{(n)}_t$:
$\color{blue} \tau_1 = \inf\{t>0 : X^{(n)}_t \not = X^{(n)}_0\}$.

Let $r^{(\mf W)}_n(i,j)$, $j\neq i\in \{1, \dots, \mf m\}$, be the
mean rate at which the trace process $Y^{\ms W_n}_t$ jumps from
$\ms W^{(i)}_n$ to $\ms W^{(j)}_n$:
\begin{equation}
\label{20}
{\color{blue}  r^{\mf W}_n(i,j)}
\; := \; \frac{1}{\pi_n(\ms W^{(i)}_n)}
\sum_{x \in\ms W^{(i)}_n} \pi_n(x) 
\sum_{y\in\ms W^{(j)}_n} R^{\mf W}_n(x,y) \;.
\end{equation}

For two sequences of positive real numbers $(\alpha_n : n\ge 1)$,
$(\beta_n : n\ge 1)$, $\color{blue} \alpha_n \prec \beta_n$ or
$\color{blue} \beta_n \succ \alpha_n$ means that
$\lim_{n\to\infty} \alpha_n/\beta_n = 0$.  The main hypothesis of the
article reads as follows:
\smallskip
\begin{itemize}
\item[\bf (H1)] There exist
$\color{blue} \mf q \ge 1$, time-scales
$\color{blue} 1 \prec \theta^{(1)}_n\prec \cdots \prec \theta^{(\mf
	q)}_n$, and partitions
$\color{blue} \mf V_p = \{\ms V^{p,1}_n, \dots, \ms V^{p, \mf n_p}_n,
\Delta^{(p)}_n\}$ of $\Xi_n$, $1\le p\le \mf q$, such that the limit
\begin{equation}
\label{01a}
{\color{blue}  r^{(p)} (i,j)} \;:=\;  \lim_{n\to\infty} \theta^{(p)}_n
\, r^{\mf V_p}_n(i, j)
\end{equation}
exists and is finite  for every $1\le p\le \mf q$,
$j\not = i\in {\color{blue} S_p := \{1, \dots, \mf n_p\}}$.
Moreover, $r^{(p)} (i,j)>0$ for some $j\neq i \in S_p$\;.
\end{itemize}

In our motivating example of the random walk in a potential field $F$, discussed
	in Section \ref{sec6}, the speeds $\theta_n^{(p)}$
	 are given by $\theta_n^{(p)}=n\, \exp(nd_p)$, 
 where $d_1<d_2<\ldots < d_{\mathfrak q}$ are the depths of valleys of  the potential $F$.

\begin{remark}
\label{r-2}
Hypothesis (H1) on the mean jump rates  \eqref{20} is condition (H0) in
\cite{bl2}. It is the main condition to be proved in order to
establish the metastable behaviour of a Markov chain by the martingale
method. It has been derived in several different contexts: for
condensing zero-range processes \cite{bl3, l2014, s2018}, inclusion
processes \cite{bdg17, kim21, ks21}, Blume-Capel, Ising and Potts
models \cite{ll16, llm19, ks21, ks21b, ks22, ks22b}.
\end{remark}

\begin{remark}
\label{r-3}
By \cite[Remark 2.9]{bl2}, for reversible dynamics the mean jump rate
$r^{\mf W}_n(i,j)$ introduced in \eqref{20} can be expressed as sums
and differences of capacities between sets of the form
$\cup_{k\in A} \ms W^{(k)}_n$ and $\cup_{\ell\in B} \ms W^{(\ell)}_n$
for disjoint subsets $A$, $B$ of $\{1, \dots, \mf m\}$. Hence, the
proof of conditions (H1) reduces to the computation of capacities.
\end{remark}

Let $\color{blue} \bb X^{(p)}_t$ be the $S_p$-valued Markov chain
induced by the rates $r^{(p)} (i,j)$.
Denote by $\color{blue} \mf R^{(p)}_1, \dots, \mf R^{(p)}_{\mf u_p}$
the closed irreducible classes of $\bb X^{(p)}_t$, and by
$\color{blue} \mf T_p$ the transient states. Clearly,
$\{\mf R^{(p)}_1, \dots $, $\mf R^{(p)}_{\mf u_p}, \mf T_p\}$ forms
a partition of the set $S_p$. 

\smallskip
\begin{itemize}
\item[\bf (H2)] For $1\le p < \mf q$, assume that the number of closed
irreducible classes of the Markov process $\bb X^{(p)}_t$ corresponds
to the number of sets $\ms V^{p+1,j}_n$ in the partition
$\mf V_{p+1}$.  In other words, we assume that
$\mf u_p = \mf n_{p+1}$, and that for $m\in S_{p+1}$,
\begin{equation}
\label{59}
\ms V^{p+1,m}_n
\;= \; \bigcup_{j\in \mf R^{(p)}_m} \ms V^{p,j}_n\;.
\end{equation}
\end{itemize}

\begin{remark}
\label{r-4}
This condition is natural and not restrictive. As
$\theta^{(p)}_n \prec \theta^{(p+1)}_n$, it states that on the longer
time-scale $\theta^{(p+1)}_n$ the Markov chain has less metastable sets
and that the metastable sets on the time-scale $\theta^{(p+1)}_n$ are
formed by unions of metastable sets on the time-scale
$\theta^{(p)}_n$.
\end{remark}

Since, by hypothesis (H1), there exists at least one pair $(j,k)$,
$k\neq j\in S_p$, such that $r^{(p)}(j,k)>0$, either $j$ is a
transient state or $j$ and $k$ belong to the same closed irreducible
class. Therefore, the number of recurrent classes for the chain
$\bb X^{(p)}_t$ (viz. $\mf u_p$ which is assumed to be equal to
$\mf n_{p+1}$) is strictly smaller than the number of $S_p$ elements
(viz. $\mf n_p$), that is $\mf n_{p+1}<\mf n_p$. In other words, the
number of sets $\ms V^{p,j}_n$ in the partition $\mf V_{p}$
strictly decreases with $p$.

Next condition states that the recurrence procedure \eqref{59} which
defines the sets $\ms V^{p,j}_n$ ends at step $\mf q$.  See Remark
\ref{rm2} below.

\smallskip
\begin{itemize}
\item [\bf (H3)] Assume that the Markov chain $\bb X^{(\mf q)}_t$ has
a unique closed irreducible class:
$\mf n_{\mf q+1}=\mf u_{\mf q} = 1$. In consequence, set
$\color{blue} S_{\mf q+1} = \{1\}$.
\end{itemize}

We extend definition \eqref{59} to $p=\mf q$, setting
\begin{equation*}
{\color{blue}  \ms V_n^{\mf q+1,1} }
\;:=\; \bigcup_{j\in \mf R^{(\mf q)}_1} \ms V_n^{\mf q,j} \;.
\end{equation*}

Recall from condition (H0) the definition of the points $x_j$,
$j\in S_1$.  Condition (H4) below asserts that the sets
$\ms V^{1,j}_n$, $j\in S_1$, are wells. Indeed, on the one hand the
sets $\ms V^{1,j}_n$ contain open balls of radius $\epsilon$, and on
the other hand the stationary state $\pi_n$ conditioned to
$\ms V^{1,j}_n$ converges to the Dirac measure $\delta_{x_j}$. 

\smallskip
\begin{itemize}
\item [\bf (H4a)] Let
$\color{blue} \ms V^{(p)}_n := \cup_{j\in S_p} \ms V^{p,j}_n$,
$1\le p\le \mf q+1$. We assume that $\pi_n(\ms V^{(1)}_n)\to 1$. Let
$\color{blue} \pi^{1,j}_n$, $j\in S_1$, be the stationary state
$\pi_n$ conditioned to $\ms V^{1,j}_n$. Then, for each $j\in S_1$,
$\pi^{1,j}_n \to \delta_{x_j}$ weakly.

\item [\bf (H4b)] There exists $\epsilon>0$ such that
$\ms V^{1,j}_n \supset B_\epsilon(x_j) \cap \Xi_n$ for all $j\in S_1$,
$n\ge 1$, where $\color{blue} B_\epsilon(x)$ stands for the ball of
radius $\epsilon$ centered at $x$. Here, $\epsilon$ is chosen
sufficiently small so that
$B_\epsilon(x_j) \cap B_\epsilon(x_k) =\varnothing$ for
$k\neq j \in S_{1}$.
\end{itemize}

For a subset $\ms A$ of $\Xi_n$, denote by $D_n(\ms A\,,\,\cdot\,)$
the Dirichlet form induced by the generator $\ms L_n$ restricted to
$\ms A$: for all $h\colon \ms A \to \bb R$,
\begin{equation}\label{eq:restricted-DF}
{\color{blue} D_n(\ms A, h)} \;:=\;
\frac{1}{2} \frac{1}{\pi_n(\ms A)} \, \sum_{x\in \ms A}
\sum_{y\in \ms A\setminus \{x\}} \pi_n(x) \, R_n(x,y)\,
[\, h(y) - h(x)\,]^2\;.
\end{equation}

Next assumption asserts that the process $X^{(n)}_t$ equilibrates much
faster inside the wells $\ms V^{1,j}$ than it moves among the
wells. The hypothesis reminds the existence of a spectral gap for the
process reflected at the boundary of $\ms V^{1,j}$, but is not quite.
The reader familiar with the path lemma in hydrodynamics will easily
estimate in concrete examples the left-hand side of the equation
below in terms of the right-hand side. We give a crude estimate in
Lemma \ref{l04}.

\smallskip
\begin{itemize}
\item[\bf (H5)] There exists a sequence $\beta_n\prec \theta^{(1)}_n$ such
that 
\begin{equation*}
\max_{x\in \ms V^{1,j}_n} \sum_{y\in \ms V^{1,j}_n} \pi^{1,j}_n(y) \, 
\{ \, h(y) \,-\,  h(x) \,\}^2
\;\le\; \beta_n\, D_n(\ms V^{1,j}_n , h)
\end{equation*}
for all $h\colon \ms V^{1,j}_n \to \bb R$, $j\in S_1$. 
\end{itemize}

\begin{remark}
\label{rm1}
Assumption (H5) is needed in Proposition \ref{l03}.  It follows from the
proof of this result that the maximum in condition (H5) can be
restricted to $x\in \ms V^{1,i}_n$ for which
$R^{\mf V_1}_n(x, \cup_{k\in S_1\setminus\{j\}} \ms V^{1,k}_n)>0$, that is
to the points $x$ which belong to the inner boundary of
$\ms V^{1,i}_n$.
\end{remark}

\begin{remark}
\label{rm4}
In the definition of metastability (cf. \cite{lrev}) one requires the
trace process $Y^{\ms V^{(p)}_n}_t$ accelerated by $\theta^{(p)}_n$ to
be well approximated by a $S_p$-valued Markov chain and  the time
spent by the process $X^{(n)}_t$ in the complement of $\ms V^{(p)}_n$
to be negligible.

The fact that $Y^{\ms V^{(p)}_n}_t$ is well approximated by a
$S_p$-valued Markov chain follows from condition (H1). The
negligibility of  the time spent in the complement of
$\ms V^{(p)}_n$ does not appear explicitly in the hypotheses, but it is
hidden in conditions (H0) and (H4a).
\end{remark}

\subsection*{Stationary measures}

Fix $1\le p\le \mf q$, and recall that we denote by
$\mf R^{(p)}_1, \dots, \mf R^{(p)}_{\mf n_{p+1}}$ the closed
irreducible classes of $\bb X^{(p)}_t$. Denote by
$\color{blue} M^{(p)}_m \in \ms P(\mf R^{(p)}_m)$, $m\in S_{p+1}$, the
stationary state of the Markov chain $\bb X^{(p)}_t$ restricted to the
closed irreducible class $\mf R^{(p)}_m$.

Let $\pi^{(1)}_j \in \ms P(\Xi)$, $j\in S_1$, be the measure given by
$\color{blue} \pi^{(1)}_j := \delta_{x_j}$, and define recursively
$\pi^{(p+1)}_j \in \ms P(\Xi)$, $1\le p\le \mf q$, $j\in S_{p+1}$, as
\begin{equation}
\label{09}
{\color{blue} \pi^{(p+1)}_j (\,\cdot\,) } \;:=\;
\sum_{k\in S_p : \ms V^{p,k}_n \subset \ms V^{p+1,j}_n}
M^{(p)}_j (k) \, \pi^{(p)}_k (\,\cdot\,)\;. 
\end{equation}
We could have represented the sum more concisely as a sum carried over
$k\in \mf R^{(p)}_j$, but it seems easier to understand its meaning by
stating that we are summing over all indices $k$ for which
$\ms V^{p,k}_n$ is a subset of $\ms V^{p+1,j}_n$.

Clearly, the measures $\pi^{(p)}_k$, $k\in S_p$, are convex
combinations of the Dirac measures $\delta_{x_i}$, $i\in S_1$. A
precise formula requires some notation.  By \eqref{59}, each set
$\ms V^{p,j}_n$, $1\le p \le \mf q$, $j\in S_p$, is the union of sets
$\ms V^{1,i}_n$, called for this reason basal sets. Let
\begin{equation*}
{\color{blue} S_{p,j} := } \big\{\, i\in S_1 : \ms V^{1,i}_n
\subset \ms V^{p,j}_n \big\}\;, \;\;  2\le p \le \mf q+1\;, \;\; j\in
S_p\;. 
\end{equation*}
Hence, $S_{p,j} \subset S_1$ represents the set of indices $i$ of the
basal sets $\ms V^{1,i}_n$ which are contained in $\ms
V^{p,j}_n$. Note that this set does not depend on $n$.

If $i$ belongs to $\cup_{j\in S_p} S_{p,j}$, then $\ms V^{1,i}_n$ is
contained in a well $\ms V^{2,k}_n$ itself contained in
$\ms V^{(p)}_n$. We denote by $\color{blue} \mf a_{2,i} \in S_2$ the
index of this well so that
$\ms V^{1,i}_n \subset \ms V^{2,\mf a_{2,i}}_n \subset \ms V^{(p)}_n$.
Similarly, for each $2\le q<p$, denote by $\color{blue} \mf a_{q,i}$
the index $k\in S_q$ of the well $\ms V^{q,k}_n$ which contains
$\ms V^{1,i}_n$ and is contained in $\ms V^{(p)}_n$:
$\ms V^{1,i}_n \subset \ms V^{q,\mf a_{q,i}}_n \subset \ms V^{(p)}_n$.
By the definition \eqref{59} of the wells,
$i \in \mf R^{(1)}_{\mf a_{2,i}}$,
$\mf a_{2,i} \in \mf R^{(2)}_{\mf a_{3,i}}, \dots, \mf a_{p-1,i} \in
\mf R^{(p-1)}_{j}$. For $i\in S_{p,j}$, let
\begin{equation}
\label{16}
{\color{blue} m_{p,j} (i)}  \;:=\; M^{(1)}_{\mf a_{2,i}}(i) \;
M^{(2)}_{\mf a_{3,i}}(\mf a_{2,i}) \; \cdots\;
M^{(p-2)}_{\mf a_{p-1,i}}(\mf a_{p-2,i} )\;
\ M^{(p-1)}_{j}(\mf a_{p-1,i})\;.
\end{equation}
It is easy to check that $m_{p,j}$ is a probability measure on
$S_{p,j}$: $m_{p,j} \in \ms P(S_{p,j})$.

At the end of Section \ref{sec3}, we show that
\begin{equation}
\label{10}
\pi^{(p)}_j (\,\cdot\,) \;=\;
\sum_{k\in S_1: \ms V^{1,k}_n \subset \ms V^{p,j}_n}
m_{p,j} (k) \,  \delta_{x_k}  (\,\cdot\,) \;, \quad
2\le p\le \mf q+1 \;,\;\; j\in S_p \;.
\end{equation}
Here again we could have represented the sum as one over $k\in
S_{p,j}$. 

Let $\color{blue} \pi^{p,j}_n$, $j\in S_p$, $1\le p\le \mf q$ be the
stationary state $\pi_n$ conditioned to $\ms V^{p,j}_n$.  By Lemma
\ref{l11}, for each $j\in S_p$, $\pi^{p,j}_n$ converges weakly to
$\pi^{(p)}_j$.

\subsection{Main results}

We are now in a position to state the main results of this article.
Fix $1\le p\le \mf q$.  Denote by $\color{blue}\bb L^{(p)}$ the
generator of the $S_p$-valued Markov chain $\bb X^{(p)}_t$. Let
$\bb I^{(p)} \colon \ms P (S_p) \to [0,+\infty)$ be the level two
large deviations rate functional of $\bb X^{(p)}_t$ given by
\begin{equation}
\label{40}
{\color{blue} \bb I^{(p)} (\omega) } \, :=\,
\sup_{\mb h} \,-\,  \sum_{j\in S_p} \omega_j \,
\frac{1}{\mb h (j)} \, (\bb L^{(p)} \mb h)(j)  \;,
\end{equation}
where the supremum is carried over all functions
$\mb h:S_p \to (0,\infty)$.  Denote by
$\ms I^{(p)} \colon \ms P(\Xi) \to [0,+\infty]$ the functional given
by
\begin{equation}
\label{o-83b}
{\color{blue} \ms I^{(p)} (\mu) } \, :=\,
\left\{
\begin{aligned}
& \bb I^{(p)} (\omega)   \quad \text{if}\;\;
\mu = \sum_{j\in S_p} \omega_j \, \pi^{(p)}_j \;\; \text{for}\;\;
\omega \in \ms P (S_p)\;,  \\
& +\infty \quad\text{otherwise}\;.
\end{aligned}
\right.
\end{equation}

The main result of the article reads as follows. Recall that
$X^{(n)}_t$, $t\ge 0$, is a $\Xi_n$-valued, reversible, irreducible,
continuous-time Markov chain, where $\Xi_n$ is the discretization of
an open and bounded subset $\Xi$ of $\bb R^d$.

\begin{theorem}
\label{mt1}
Assume that conditions {\rm (H0) -- (H5)} are in force. Then, a full
$\Gamma$-expansion of $\ms I_n$ as in Definition \ref{def1} is given
by the speeds $(\theta^{(p)}_n, n\ge 1)$, $1\le p\le \mf q$, and
functionals $\ms I^{(p)}\colon \ms P(\Xi) \to [0, +\infty]$,
$0\le p\le \mf q$, appearing in conditions {\rm(H0)} and {\rm (H1)},
and \eqref{o-83b}.
\end{theorem}

This theorem  provides an expansion of the
large deviations rate functional $\ms I_n$ which can be written as in
\eqref{f05}.  Therefore, the rate functional $\ms I_n$ encodes all the
characteristics of the metastable behaviour of the chain
$X^{(n)}_t$. The time-scales $\theta^{(p)}_n$ appear as the speeds of
the expansion, and, by \eqref{o-83b} and Lemma \ref{l39}, the convex
hull of the metastable states $\pi^{(p)}_j$, $j\in S_p$, form the
$0$-level set of the rate functional $\ms I^{(p-1)} (\,\cdot\,)$.

\begin{remark}
\label{rm-1}
For many dynamics the $0$-level set of the first term(s) of the
expansion of rate functional $\ms I_n$ are not convex combinations of
Dirac measures.  More precisely, there may exist speeds
$(\kappa^{(p)}_n, n\ge 1)$, $1\le p\le \mf r$, and functionals
$\ms J^{(p)}\colon \ms P(\Xi) \to [0, +\infty]$, $0\le p\le \mf r$
satisfying conditions (a) -- (c) of Definition \ref{def1} and such
that
\begin{itemize}
\item [(d')] There exists a finite set
$\ms M = \{x_1, \dots , x_{\mf n_1}\}$ such that
$\{ \mu \in \ms P(\Xi) : \ms J^{(\mf r)} (\mu)=0\} = \{\sum_{j\in S_1}
\omega_j\, \delta_{x_j} : \omega\in \ms P(S_1)\}$, where
$S_1 = \{1, \dots, \mf n_1\}$.
\end{itemize}

Condition (d') states that the rate functional
$\ms K_n : = \kappa^{(\mf r)}_n\, \ms I_n$ satisfies condition (H0).
In view of \eqref{f4}, the functional $\ms K_n $ is the level two
large deviations rate functional of the Markov chain induced by the
generator $\kappa^{(\mf r)}_n\, \ms L_n$, that is, the rate functional
associated to the process $X^{(n)}_t$ speeded-up by
$\kappa^{(\mf r)}_n$. Denote this process by $Y^{(n)}_t$:
$Y^{(n)}_t = X^{(n)}_{t \,\kappa^{(\mf r)}_n}$.

If the process $Y^{(n)}_t$ and its associated rate-functional $\ms
K_n$ also fulfill hypotheses (H1) -- (H5), by Theorem \ref{mt1},
\begin{equation*}
\ms K_n  \;=\; \ms J^{(\mf r)} \;+\; \sum_{p=1}^{\mf q}
\frac{1}{\theta^{(p)}_n} \, \ms I^{(p)} \; \cdot
\end{equation*}
Hence, as $\ms K_n : = \kappa^{(\mf r)}_n\, \ms I_n$, the full
$\Gamma$-expansion of $\ms I_n$ reads as
\begin{equation}
\label{f05-general}
\ms I_n  \;=\; \ms J^{(0)} \;+\; \sum_{p=1}^{\mf r}
\frac{1}{\kappa^{(p)}_n} \, \ms J^{(p)} 
\;+\; \sum_{q=1}^{\mf q}
\frac{1}{\kappa^{(\mf r)}_n\, \theta^{(q)}_n} \, \ms I^{(q)} \; .
\end{equation}

This article presents a general method to derive the last terms
$\ms I^{(q)}$, $1\le q\le \mf q$, of the expansion. The analysis of
the first ones $\ms J^{(p)}$, $1\le p\le \mf r$, has to be carried out
by other methods. The example of random walks in a potential field,
discussed in Section \ref{sec6}, illustrates and clarifies this
remark. There, $\mf r=1$, and the first two terms of the expansion
$\ms J^{(0)}$ and $\ms J^{(1)}$ are obtained by a direct computation.
\end{remark}

\smallskip

For some models, as condensing zero-range processes \cite{bl3, l2014,
s2018}, in order to compute the limit of the mean jump rates \eqref{20}, one needs to take
microscopic wells, and condition (H4b) fails. Indeed, if in
Propositions 4.1 and 5.1 of \cite{bl3} one replaces microscopic
neighbourhoods by small macroscopic ones, say balls of radius
$\epsilon$, the capacities between sets, and in consequence the jump
rates, become $\epsilon$-dependent and condition (H4b) fails.  In
Theorem \ref{mt2}, we replace this condition by (H6).

\smallskip
\begin{itemize}
\item [\bf (H6a)] There exists $c_0>0$ such that for each
$1\le p \le \mf q$, $j\in S_p$,
$\pi_n(\ms V^{p,j}_n) / \pi_n(\ms V^{(p)}_n) \ge c_0$ for all $n\ge 1$.

\item [\bf (H6b)] Fix $1\le p \le \mf q$. Let $(\mu_n:n\ge 1)$ be
a sequence in $\ms P(\Xi_n)$ such that
$\mu_n$ converges weakly to
$\sum_{j\in S_1} \omega_j \, \delta_{x_j}$ for some
$\omega\in \ms P(S_1)$, and which fulfills
\begin{equation}
\label{02}
\liminf_{n\to\infty} \, \theta^{(p)}_n\, \ms I_n(\mu_n) \;<\;
\infty\;.
\end{equation}
Then, for all $j\in S_1$,
$\displaystyle \lim_{n\to\infty} \mu_n (\ms V^{1,j}_n) \;=\;
\omega_j$.
\end{itemize}

\begin{remark}
\label{rm5}
Condition (H6a) is only used in Lemma \ref{l14} to show that
$Z_n\to 1$. In concrete examples one can drop this condition and prove
directly that $Z_n$ converges to $1$. 
\end{remark}

\begin{remark}
\label{rm5b}
Here is a strategy to prove condition (H6b) assuming that the
sequences of Markov chains satisfy two conditions.  Let
$\pi^{j,\epsilon}_n$, $j\in S_1$, $\epsilon>0$, be the stationary
state $\pi_n$ conditioned to $B_\epsilon (x_j)$. Suppose that
$\pi^{j,\epsilon}_n (B_\epsilon (x_j) \setminus \ms V^{1,j}_n) \to 0$.
This is the first condition which has to be proved in specific
examples and which holds for condensing zero range processes.  Since
$\mu_n$ converges weakly to
$\sum_{j\in S_1} \omega_j \, \delta_{x_j}$,
$\mu_n(B_\epsilon (x_j)) \to \omega_j$. To complete the proof of
(H6b), it remains to show that
$\mu_n (B_\epsilon (x_j) \setminus \ms V^{1,j}_n) \to 0$.  One can
estimate the difference
$\mu_n (B_\epsilon (x_j) \setminus \ms V^{1,j}_n) - \pi^{j,\epsilon}_n
(B_\epsilon (x_j) \setminus \ms V^{1,j}_n)$ by a constant times
$\ms I_n(\mu_n)$ (using the so-called path lemmas in hydrodynamic
limit theory, cf. the proof of \eqref{05} which relies on the same
ideas). If the constant is not too large, and this is the second
condition alluded to at the beginning of the remark, it follows from
the previous bound and \eqref{02} that
$\mu_n (B_\epsilon (x_j) \setminus \ms V^{1,j}_n) - \pi^{j,\epsilon}_n
(B_\epsilon (x_j) \setminus \ms V^{1,j}_n) \to 0 $.
\end{remark}

\begin{theorem}
\label{mt2}
Assume that conditions {\rm (H0) -- (H3), (H4a), (H5), (H6)} are in
force. Then, the assertion of Theorem \ref{mt1} holds.
\end{theorem}

\begin{remark}
\label{rm2}
Since, by hypothesis (H5), $\bb X^{(q)}_t$ has a unique closed
irreducible class, the $0$-level set of $\ms I^{(\mf q)}$ is a
singleton whose element is the measure
$\sum_{j\in S_{\mf q}} M^{(\mf q)}_1 (j) \, \pi^{(\mf q)}_j$. This
means that  we arrived at the end of the expansion of $\ms I_n$
with the functional $\ms I^{(\mf q)}$.
\end{remark}

Next result is a simple consequence of the level two large deviations
principle \eqref{50} and the $\Gamma$-convergence stated in the
previous theorems and in condition (H0) for $p=0$.  (cf. Corollary 4.3
in \cite{mar}). Recall from \eqref{48} the definition of $L_t^{(n)}$.

\begin{corollary}
\label{cor2}
Suppose that the hypotheses of Theorem \ref{mt1} or \ref{mt2} hold.
Fix $0\le p\le \mf q$ and set $\theta^{(0)}_n=1$.  For every
closed subset $F$ and open subset $G$ of $\ms P(\Xi)$,
\begin{align*}
\begin{aligned}
\limsup_{n\to \infty} \, \limsup_{t\to \infty}
\, \frac{\theta^{(p)}_n}{t} \,\max_{x\in \Xi_n} \, \log\, 
\mb P^{(n)}_{\! x} \Big[\, L_t^{(n)}
\,\in\, F \, \Big] \;&\le\; -\, \inf_{\mu\in F} \ms I^{(p)} (\mu) \;,
\\
\liminf_{n\to \infty} \, \liminf_{t\to \infty}
\, \frac{\theta^{(p)}_n}{t}  \, \min_{x\in \Xi_n} \,\log\, 
\mb P^{(n)}_{\! x} \Big[\, L_t^{(n)}
\,\in\, G \, \Big] \;&\ge\; -\, \inf_{\mu\in G} \ms I^{(p)} (\mu) \;.
\end{aligned}
\end{align*}
\end{corollary}

\smallskip\noindent{\bf Outline of the proof.} The proof of the
\textit{$\Gamma$-liminf}, presented in Lemma \ref{l06}, follows the method
introduced in \cite{bgl-lm}. It relies on the variational formula
\eqref{f4} for the rate functional and on the fact that the generator
applied to the harmonic extension of a function coincides with the
generator of the trace process applied to the original function. The
main difficulty in this proof is handled in Proposition \ref{l03}, the
principal result of Section \ref{sec4}.

The proof of the \textit{$\Gamma$-limsup}, presented in Lemma \ref{l07}, is
more demanding. As in \cite{bgl-lm}, one first proves the
\textit{$\Gamma$-limsup} of the trace process. This is the content of the next
section and stated in Proposition \ref{l16}. It is in the proof of
this result that the metastable behavior of the process, described in
assumptions (H1), (H2), and the asymptotic behavior of the stationary
state, presented in condition (H4a), are crucial.

After this initial step, one turns to the proof of the
\textit{$\Gamma$-limsup}, which is carried out inductively. Condition (H0) and
the induction hypothesis permit to restrict the proof of the upper
bound to probability measures $\mu$ which are convex combinations of
Dirac masses. By the proof of the \textit{$\Gamma$-limsup} for the trace
process, there exists a sequence $\mu_n$, whose support is
concentrated on the set on which the trace is taken, which converges
to $\mu$ and satisfies the \textit{$\Gamma$-limsup}. All the difficulty
consists in extending the measure $\mu_n$ to a measure $\nu_n$ defined
on the whole space and whose energy is close to the trace-energy of
$\mu_n$. This step requires delicate estimates and is presented in
Lemma \ref{l14}.

\smallskip \noindent
\textbf{Organization of the paper.}
The rest of the paper is organized as follows. In Section \ref{sec3}
we prove the \emph{$\Gamma$-limsup} for the trace process on
$\ms V^{(p)}_n$.  In Section \ref{sec4}, we replace $\pi_n$ in
equation \eqref{20} by a measure $\mu_n$, and identify the limit as in
\eqref{01a}, provided that the sequence $(\mu_n:n\ge 1)$ fulfills
\eqref{02}. This is the main step in the proof of the
\emph{$\Gamma$-liminf}. The proofs of Theorems \ref{mt1} and \ref{mt2}
are presented in Section \ref{sec5}. In Section \ref{sec6}, to
illustrate the theory, we show that random walks on a potential field
satisfying a natural set of hypotheses fulfill the conditions of
Theorem \ref{mt1}.

\section{\emph{$\Gamma$-limsup} of the trace}
\label{sec3}

Fix $1\le p\le \mf q$, and denote by
$\ms I^{(p)}_n \colon \ms P(\ms V_n^{(p)}) \to [0,+\infty)$ the
occupation time large deviations rate functional of the trace process
$Y^{\ms V^{(p)}_n}_t$:
\begin{equation}
\label{eq:2}
{\color{blue}  \ms I^{(p)}_n (\mu)} \;:=\;
\sup_{u} \,-\,  \sum_{x\in \ms V_n^{(p)}} \mu(x)\,
\frac{1}{u(x)} \, [\, (\mf T_{\ms V_n^{(p)}} \ms L_n) \, u )\,]
(x)  \;,
\end{equation}
where the supremum is carried over all functions
$u\colon \ms V_n^{(p)}\to (0,\infty)$ and
$\color{blue} \mf T_{\ms V_n^{(p)}} \ms L_n$ is the generator of the
trace process $Y^{\ms V^{(p)}_n}_t$.  The main result of this section,
Proposition \ref{l16}, states that $\ms I^{(p)}$ is a 
\emph{$\Gamma$-limsup}
for the sequence $\theta^{(p)}_n\, \ms I^{(p)}_n$.

Denote by $\color{blue} R^p_n(\,\cdot\,,\,\cdot\,)$ the jump rates of
the Markov chain $Y^{\ms V^{(p)}_n}_t$. Equation \eqref{o-40} provides
an explicit formula for $R^p_n(x,y)$. Let
$\color{blue} R^p_n(x, \ms A) \,=\, \sum_{y\in\ms A} R^p_n(x,y)$,
$\ms A \subset \ms V^{(p)}_n$.  We start with estimates on the
measures of the wells $\ms V^{p,j}_n$, $j\in S_p$.

\begin{lemma}
\label{l09}
Assume that condition {\rm (H1)} is in force.  Fix $1\le p\le \mf
q$. If $r^{(p)}(j,k)>0$ for some $k\not = j\in S_p$, then
$\pi_n(\ms V^{p,j}_n) / \pi_n(\ms V^{p,k}_n)$ converges to
$r^{(p)}(k,j)/r^{(p)}(j,k) \in [0,+\infty)$.
\end{lemma}

\begin{proof}
By reversibility,
\begin{equation}
\label{08b}
\sum_{x\in \ms V^{p,j}_n} \pi_n (x) \,
R^p_n(x, \ms V^{p,k}_n) \;=\;
\sum_{x\in \ms V^{p,k}_n}
\pi_n (x)  \, R^p_n(x, \ms V^{p,j}_n)\;.
\end{equation}
Divide and multiply the left-hand side by $\pi_n(\ms V^{p,j}_n)$ and
divide both sides by $\pi_n(\ms V^{p,k}_n)$.  By condition (H1), and
since $r^{(p)}(j,k)>0$, the sequence
$\pi_n(\ms V^{p,j}_n)/\pi_n(\ms V^{p,k}_n)$ converges to
$r^{(p)}(k,j)/r^{(p)}(j,k)\in [0,+\infty)$, as claimed.
\end{proof}

Recall from equation \eqref{09} the definition of the measures
$M^{(p)}_m \in \ms P(\mf R^{(p)}_m)$, $m\in S_{p+1}$.

\begin{corollary}
\label{l10}
Assume that conditions {\rm (H1), (H2)} are in force.  Fix
$1\le p\le \mf q$ and $m\in S_{p+1}$.  Then, for all $j\in S_p$ such
that $\ms V^{p,j}_n \subset \ms V^{p+1,m}_n$, the sequence
$\pi_n(\ms V^{p,j}_n) / \pi_n(\ms V^{p+1,m}_n)$ converges to
$M^{(p)}_m(j)$.
\end{corollary}

\begin{proof}
Fix $1\le p\le \mf q$, $m\in S_{p+1}$ and $j\in S_p$ such that
$\ms V^{p,j}_n \subset \ms V^{p+1,m}_n$.  By definition of
$\ms V^{p+1,m}_n$, $j\in \mf R^{(p)}_{m}$. Since $\mf R^{(p)}_{m}$ is
a closed irreducible class of $\bb X^{(p)}_t$, by Lemma \ref{l09},
for all $k, j \in \mf R^{(p)}_{m}$,
$\pi_n(\ms V^{p,j}_n) / \pi_n(\ms V^{p,k}_n)$ converges to a real
number in $(0,\infty)$. In particular, for $j \in \mf R^{(p)}_{m}$,
$\pi_n(\ms V^{p,j}_n) / \pi_n(\ms V^{p+1,m}_n)$ converges to a real
number denoted by $\bs m(j) \in (0,1)$. By \eqref{08b},
\begin{equation}
\label{08}
\pi_n(\ms V^{p,j}_n)\, 
\sum_{x\in \ms V^{p,j}_n} \frac{\pi_n (x)}{\pi_n(\ms V^{p,j}_n)} \,
R^p_n(x, \ms V^{p,k}_n) \;=\; \pi_n(\ms V^{p,k}_n)\,
\sum_{x\in \ms V^{p,k}_n} \frac{\pi_n (x)}{\pi_n(\ms V^{p,k}_n)} \,
R^p_n(x, \ms V^{p,j}_n)\;.
\end{equation}
Dividing both sides by $\pi_n(\ms V^{p+1,m}_n)$ and taking the limit,
condition (H1) yields that $\bs m(\, \cdot\,)$ satisfies the detailed
balance conditions for the rates $r^{(p)}$ restricted to
$\mf R^{(p)}_{m}$, so that $\bs m(j) = M^{(p)}_m(j)$, as claimed.
\end{proof}

Recall from \eqref{16} the definition of the probability measure
$m_{p,j} \in \ms P(S_{p,j})$.

\begin{corollary}
\label{l10b}
Assume that conditions {\rm (H1), (H2)} are in force. 
Fix $1\le p\le \mf q$ and $j\in S_{p}$.  Then, for all $i\in S_{p,j}$
the sequence $\pi_n(\ms V^{1,i}_n) / \pi_n(\ms V^{p,j}_n)$ converges
to $m_{p,j} (i) \in (0,1)$.
\end{corollary}

\begin{proof}
Rewrite the ratio $\pi_n(\ms V^{1,i}_n) / \pi_n(\ms V^{p,j}_n)$ as
\begin{equation*}
\frac{\pi_n(\ms V^{1,i}_n)}{\pi_n(\ms V^{2,\mf a_{2,i}}_n)} \,
\frac{\pi_n(\ms V^{2,\mf a_{2,i}}_n)} {\pi_n(\ms V^{3,\mf a_{3,i}}_n)}
\;\cdots\;
\frac{\pi_n(\ms V^{p-1,\mf a_{p-1,i}}_n)} {\pi_n(\ms V^{p,j}_n)}\;\cdot
\end{equation*}
By Corollary \ref{l10}, this expression converges to $m_{p,j} (i)$, as
claimed. 
\end{proof}

Recall from \eqref{59} the definition of the sets $\ms V^{p+1,m}_n$,
$m\in S_{p+1}$. Since we want
$\mf V_{p+1} = \{\ms V^{p+1,1}_n, \dots, \ms V^{p+1, \mf n_{p+1}}_n,
\Delta^{(p+1)}_n\}$ to form a partition of $\Xi_n$, we define
$\color{blue} \Delta^{(p+1)}_n$ as
$\color{blue} \Delta^{(p+1)}_n := \Delta^{(p)}_n \cup_{j\in \mf T_p}
\ms V^{p,j}_n$.

\begin{lemma}
\label{l17}
Assume that conditions {\rm (H1), (H2), (H4a)} are in force.  Then,
for all $1\le p \le \mf q$, $\pi_n(\ms V^{(p)}_n) \to 1$.
\end{lemma}

\begin{proof}
The proof is by induction on $p$. For $p=1$, it is hypothesis (H4a).
Suppose that it holds for all $1\le p'\le p$ so that
$\pi_n(\Delta^{(p)}_n)\to 0$. It remains to show that $\pi_n(\ms
V^{p,j}_n)\to 0$ for all $j\in \mf T_p$.

Fix $j\in \mf T_p$. There exist
$j=j_0, j_1, \dots, j_r = k \in \mf R^{(p)}_i$ for some $i\in S_{p+1}$
such that $r^{(p)}(j_a, j_{a+1})>0$ for $0\le a<r$. By Lemma
\ref{l09}, as $j\in \mf T_p$ and $k \in \mf R^{(p)}_i$, $\pi_n(\ms
V^{p,j}_n)/ \pi_n(\ms V^{p,k}_n) \to 0$, so that $\pi_n(\ms
V^{p,j}_n) \to 0$, as claimed.
\end{proof}

Fix $1\le p\le \mf q$. Denote by $\color{blue} \mf Q^{(p)}_a$,
$1\le a \le \ell_p$, the equivalent classes of the Markov chain
$\bb X^{(p)}_t$, by $\color{blue} \mf D^{(p)}_a$, $1\le a \le m_{p}$,
the ones which are not singletons, and by $\color{blue} S^{\rm sgl}_p$
the set of states $j\in S_p$ such that $\{j\}$ is an equivalent
class. Clearly,
\begin{equation}
\label{27}
S_p \;=\; \bigcup_{a=1}^{\ell_p} \mf
Q^{(p)}_a \;=\; S^{\rm sgl}_p \;\cup\; \bigcup_{a=1}^{m_p} \mf
D^{(p)}_a\;.
\end{equation}
Next result extends Corollary \ref{l10} to equivalent classes which
are not closed. Let
$\color{blue} \ms W^{p,a}_n = \cup_{j\in \mf D^{(p)}_a} \ms
V^{p,j}_n$, $1\le a\le m_p$, and let $\color{blue} m^{p,a}_j$ be the
unique stationary state (actually reversible) of the chain
$\bb X^{(p)}_t$ reflected at $\mf D^{(p)}_a$.

\begin{lemma}
\label{l08}
Assume that conditions {\rm (H1), (H2)} are in force.  For all
$1\le a\le m_p$, $j\in \mf D^{(p)}_a$,
$\pi_n(\ms V^{p,j}_n)/\pi_n(\ms W^{p,a}_n) \to m^{p,a}_j \in (0,1)$.
\end{lemma}

The proof of this result is similar to the one of Corollary \ref{l10}
and left to the reader.  Recall the definition of the measures
$\pi^{(p)}_j$, $j\in S_p$, introduced in \eqref{09}.

\begin{lemma}
\label{l11}
Assume that conditions {\rm (H1), (H2) and (H4a)}  are in force.
Let $\pi^{p,j}_n$, $j\in S_p$, $1\le p\le \mf q$ be the stationary
state $\pi_n$ conditioned to $\ms V^{p,j}_n$. Then, for each
$j\in S_p$, $\pi^{p,j}_n \to \pi^{(p)}_j$ weakly.
\end{lemma}

\begin{proof}
The proof is by induction on $p$. For $p=1$, it is hypothesis (H4a).
Assume it has been proven for $1\le r\le p$.
By construction,
$\ms V^{p+1,m}_n = \cup_{k\in \mf R^{(p)}_m} \ms V^{p,k}_n$.  By
definition of $\pi^{p+1,m}_n$,
\begin{equation*}
\pi^{p+1,m}_n (\,\cdot\,) \;=\;
\sum_{k\in \mf R^{(p)}_m}
\frac{\pi_n(\ms V^{p,k}_n)}{\pi_n(\ms V^{p+1,m}_n)} \;
\pi^{p,k}_n(\,\cdot\,) \;.
\end{equation*}
By Corollary \ref{l10}, $\pi_n(\ms V^{p,k}_n)/\pi_n(\ms V^{p+1,m}_n)$
converges to $M^{(p)}_m(k)$, and by the induction assumption,
$\pi^{p,k}_n$ converges weakly to $\pi^{(p)}_k$. Hence,
$\pi^{p+1,m}_n$ converges weakly to 
\begin{equation*}
\sum_{k\in \mf R^{(p)}_m}
M^{(p)}_m(k) \;
\pi^{(p)}_k(\,\cdot\,) \;,
\end{equation*}
which, by \eqref{09}, is equal to $\pi^{(p+1)}_m$, as claimed.
\end{proof}

Fix a measure $\mu = \sum_{j\in S_p} \omega_j \, \pi^{(p)}_j$ for some
$\omega\in \ms P(S_p)$ such that $\omega_j >0$ for all $j\in S_p$. Let
$\color{blue} \omega^{(a)}_j = \omega_j/ \Omega_a$,
$\color{blue} \Omega_a = \sum_{j\in \mf D^{(p)}_a} \omega_j$,
$1\le a \le m_p$.  Denote by $\color{blue} \bb L^{(p)}_a $ the
generator of the Markov chain $\bb X^{(p)}_t$ reflected at
$\mf D^{(p)}_a$.  By \eqref{o-83b}, Lemma A.7 and equation (A-14) in
\cite{l-mld},
\begin{equation}
\label{18}
\ms I^{(p)} (\mu) \, =\, \sum_{a=1}^{m_p}
\Omega_a\, I_{\bb L^{(p)}_a} (\omega^{(a)}) 
\;+\;
\sum_{a=1}^{\ell_p} \sum_{b\neq a}
\sum_{j\in \mf Q^{(p)}_a} \sum_{k\in \mf Q^{(p)}_b} \omega_j \,
r^{(p)}(j,k)  \;,
\end{equation}
where $ I_{\bb L^{(p)}_a} (\,\cdot\,)$ represents the level two large
deviations rate functional associated to the generator $\bb L^{(p)}_a$.
Note that we can restrict the first sum in the second term of the
right-hand side to the transient equivalent classes.  The main result
of this section reads as follows.

\begin{proposition}
\label{l16}
Assume that conditions {\rm (H1), (H2) and (H4a)} hold.  Fix
$1\le p\le \mf q$. For $\omega \in \ms P(S_p)$, let
$\mu_n \in \ms P(\ms V^{(p)}_n)$ be the measure given by
\begin{equation}
\label{26}
\mu_n(\,\cdot\,) \;=\; \sum_{j\in S_p} \omega_j \, \pi^{p,j}_n
(\,\cdot\,) \;.
\end{equation}
Assume that $\omega_j>0$ for all $j\in S_p$. Then,
$\mu_n\to \mu := \sum_{j\in S_p} \omega_j \, \pi^{(p)}_j$ and
\begin{equation*}
\limsup_{n\to\infty} \theta^{(p)}_n\, \ms I^{(p)}_n(\mu_n) \;\le\; \ms
I^{(p)}(\mu)\;.  
\end{equation*}
\end{proposition}

\begin{proof}
Fix $1\le p\le \mf q$, and let $\pi^{(p)}_n \in \ms P(\ms V^{(p)}_n)$ be
the measure given by
$\color{blue} \pi^{(p)}_n (x) := \pi_n(x) / \pi_n(\ms V^{(p)}_n)$.  By
Lemma \ref{l11}, $\mu_n$ converges to $\mu$ in the weak topology.  On
the other hand, as by \cite[Proposition 6.3]{bl2} the trace process is
reversible, and since
\begin{equation*}
\frac{d\, \mu_n}{d\, \pi^{(p)}_n} \, (x) \;=\;
\frac{\omega_j }{\pi_n(\ms V^{p,j}_n)}\, \pi_n(\ms V^{(p)}_n)\;, \quad
x\,\in\, \ms V^{p,j}_n\;,
\end{equation*}
by formula \eqref{f6} for the trace process and elementary
simplifications,
\begin{equation*}
\ms I^{(p)}_n(\mu_n) \;=\; \frac{1}{2}\,
\sum_{j\in S_p}\sum_{k\not =j} \Big(
\sqrt{\frac{\omega_k }{\pi_n(\ms V^{p,k}_n)}}
\,-\, \sqrt{\frac{\omega_j }{\pi_n(\ms V^{p,j}_n)}} \Big)^2
\sum_{x\in \ms V^{p,j}_n} \pi_n (x) \, R^p_n(x, \ms V^{p,k}_n)\;.
\end{equation*}

By Lemma \ref{l08} and condition (H1), for $1\le a \le m_p$. 
\begin{equation*}
\begin{aligned}
& \lim_{n\to \infty} \frac{\theta^{(p)}_n}{2}\,
\sum_{j\in \mf D^{(p)}_a}\sum_{k\in \mf D^{(p)}_a \setminus \{j\}}
\Big( \sqrt{\frac{\omega_k }{\pi_n(\ms V^{p,k}_n)}}
\,-\, \sqrt{\frac{\omega_j }{\pi_n(\ms V^{p,j}_n)}} \Big)^2
\sum_{x\in \ms V^{p,j}_n} \pi_n (x) \, R^p_n(x, \ms V^{p,k}_n) \\
&\qquad =\;
\frac{1}{2}\,
\sum_{j\in \mf D^{(p)}_a}\sum_{k\in \mf D^{(p)}_a \setminus \{j\}}
\Big( \sqrt{\frac{\omega_k }{m^{p,a}_k}}
\,-\, \sqrt{\frac{\omega_j }{m^{p,a}_j}} \Big)^2
m^{p,a}_j \,  r^{(p)}(j,k) \;=\; \Omega_a\, I_{\bb L^{(p)}_a} (\omega^{(a)}) \;.
\end{aligned}
\end{equation*}
This expression corresponds to the first term in \eqref{18}.

It remains to consider the case where $j$ and $k$ belongs to different
equivalent classes, so that either $r^{(p)}(j,k) = 0$ or
$r^{(p)}(k,j) = 0$. Assume first that $r^{(p)}(j,k) > 0$ and
$r^{(p)}(k,j) = 0$. By \eqref{08} and condition (H1),
$\pi_n(\ms V^{p,j}_n)/ \pi_n(\ms V^{p,k}_n) \to 0$. Hence, dividing
and multiplying next expression by $\pi_n(\ms V^{p,j}_n)$ yields, by
condition (H1), that
\begin{equation*}
\lim_{n\to \infty} \theta^{(p)}_n\,
\Big( \sqrt{\frac{\omega_k }{\pi_n(\ms V^{p,k}_n)}}
\,-\, \sqrt{\frac{\omega_j }{\pi_n(\ms V^{p,j}_n)}} \Big)^2
\sum_{x\in \ms V^{p,j}_n} \pi_n (x) \, R^p_n(x, \ms V^{p,k}_n)
\;=\; \omega_j\, r^{(p)}(j,k) \;.
\end{equation*}
Similarly, by the detailed-balance relation \eqref{08b},
\begin{equation*}
\lim_{n\to \infty} \theta^{(p)}_n\,
\Big( \sqrt{\frac{\omega_k }{\pi_n(\ms V^{p,k}_n)}}
\,-\, \sqrt{\frac{\omega_j }{\pi_n(\ms V^{p,j}_n)}} \Big)^2
\sum_{x\in \ms V^{p,k}_n} \pi_n (x) \, R^p_n(x, \ms V^{p,j}_n)
\;=\; \omega_j\, r^{(p)}(j,k) \;.
\end{equation*}
These expressions correspond to the second term in \eqref{18}.

Finally, we consider the case where $r^{(p)}(j,k) = r^{(p)}(k,j) =
0$. Let $n'$ be a subsequence such that
\begin{equation*}
\begin{aligned}
& \limsup_{n\to \infty} \theta^{(p)}_n\,
\Big( \sqrt{\frac{\omega_k }{\pi_n(\ms V^{p,k}_n)}}
\,-\, \sqrt{\frac{\omega_j }{\pi_n(\ms V^{p,j}_n)}} \Big)^2
\sum_{x\in \ms V^{p,j}_n} \pi_n (x) \, R^p_n(x, \ms V^{p,k}_n) \\
& \quad \;=\;
\lim_{n'\to \infty} \theta^{(p)}_{n'}\,
\Big( \sqrt{\frac{\omega_k }{\pi_{n'}(\ms V^{p,k}_{n'})}}
\,-\, \sqrt{\frac{\omega_j }{\pi_{n'}(\ms V^{p,j}_{n'})}} \Big)^2
\sum_{x\in \ms V^{p,j}_{n'}} \pi_{n'} (x) \, R^p_{n'}(x, \ms V^{p,k}_{n'})
\;.
\end{aligned}
\end{equation*}
Assume that
$\pi_{n'}(\ms V^{p,j}_{n'})/ \pi_{n'}(\ms V^{p,k}_{n'}) \le 1$ for
infinitely many ${n'}$'s. Denote by $n''$ the subsequence for which
this inequality holds always. Dividing and multiplying the previous
expression by $\pi_{n''}(\ms V^{p,j}_{n''})$, by condition (H1), as
$r^{(p)}(j,k) =0$ and
$\pi_{n''}(\ms V^{p,j}_{n''})/ \pi_{n''}(\ms V^{p,k}_{n''}) \le 1$,
\begin{equation*}
\limsup_{n''\to \infty} \theta^{(p)}_{n''}\,
\Big( \sqrt{\frac{\omega_k }{\pi_{n''}(\ms V^{p,k}_{n''})}}
\,-\, \sqrt{\frac{\omega_j }{\pi_{n''}(\ms V^{p,j}_{n''})}} \Big)^2
\sum_{x\in \ms V^{p,j}_{n''}} \pi_{n''} (x) \,
R^p_{n''}(x, \ms V^{p,k}_{n''}) \;=\; 0 \;.
\end{equation*}

In the case where
$\pi_{n'}(\ms V^{p,j}_{n'})/ \pi_{n'}(\ms V^{p,k}_{n'}) \ge 1$ for
infinitely many ${n'}$'s, we use the detailed balance condition
\eqref{08b} to replace
\begin{equation*}
\sum_{x\in \ms V^{p,j}_{n'}} \pi_{n'} (x) \, R^p_{n'}(x, \ms
V^{p,k}_{n'})
\text{\; by\;}
\sum_{x\in \ms V^{p,k}_{n'}} \pi_{n'} (x) \, R^p_{n'}(x, \ms
V^{p,j}_{n'})
\end{equation*}
and repeat the argument presented above.

The assertion of the proposition follows from the previous estimates
and formula \eqref{18} for $\ms I^{(p)} (\mu)$.
\end{proof}

\begin{remark}
In the previous proof, we only used the hypothesis that $\omega_j>0$
for all $j\in S_p$ to be able to use formula \eqref{18} for $\ms
I^{(p)}(\mu)$. 
\end{remark}

We conclude this section proving identity \eqref{10} by induction.
For $p=2$ the equality holds because the sums in \eqref{10} for $p=2$
and in \eqref{09} for $p=1$ are performed over the same set of
indices, and $\pi^{(1)}_k = \delta_{x_k}$.
$M^{(1)}_j (k) = m_{2,j}(k)$. Suppose that \eqref{10} holds for all
$2\le p\le r$. By \eqref{09} and by \eqref{10} for $p=r$, for all
$j\in S_{r+1}$,
\begin{equation*}
\pi^{(r+1)}_j (\,\cdot\,)  \;=\;
\sum_{k\in S_r : \ms V^{r,k}_n \subset \ms V^{r+1,j}_n} \;
\sum_{\ell \in S_1: \ms V^{1,\ell}_n \subset \ms V^{r,k}_n}
M^{(r)}_j (k) \,
m_{r,k} (\ell) \,  \delta_{x_\ell}  (\,\cdot\,) \;. 
\end{equation*}
The sums can be merged as a sum over all indices $\ell \in S_1$ such
that $\ms V^{1,\ell}_n \subset \ms V^{r+1,j}_n$. This obvious
conclusion can be reached due to the representation of the sums as
ones over indices of sets contained in other sets. On the other hand,
by \eqref{16}, $M^{(r)}_j (k) \, m_{r,k} (\ell)  \,=\, m_{r+1,j}
(\ell)$. Hence, the previous sum is equal to
\begin{equation*}
\pi^{(r+1)}_j (\,\cdot\,)  \;=\;
\sum_{\ell \in S_1: \ms V^{1,\ell}_n \subset \ms V^{r+1,j}_n}
m_{r+1,j} (\ell) \,  \delta_{x_\ell}  (\,\cdot\,) \;,
\end{equation*}
which is \eqref{10} for $r+1$. This completes the proof of the claim.

\section{Rate functional estimates}
\label{sec4}

Fix $1\le p\le \mf q$. In the proof of the \emph{$\Gamma$-liminf} of
$\theta^{(p)}_n\, \ms I_n$, we may restrict our attention to sequences
of measures $(\mu_n:n\ge 1)$ which fulfills condition \eqref{02}.  The
main result of this section, Proposition \ref{l03}, states that for
such measures we recover the limit in \eqref{01a} if we replace in
\eqref{20} the conditional measure $\pi_n (\cdot)/\pi_m(\ms V^{p,j})$
by the measure $\mu_n$.

The proof of this proposition relies on two simple lemmata. The first
one, Lemma \ref{l02}, asserts that the energy of a function on the
union of disjoint subsets of $\Xi_n$ is bounded by the total
energy. The second one, Lemma \ref{l15}, states that under condition
(H4b) the measure $\mu_n (\ms V^{1,j}_n) $ of a valley $\ms V^{1,j}_n$
converges to $\omega_j$ for any sequence of measures $\mu_n$ which
converges weakly to $\sum_{j\in S_1} \omega_j \, \delta_{x_j}$ for
some $\omega\in \ms P(S_1)$.

Recall from \eqref{eq:restricted-DF} the definition of
the Dirichlet form $D_n(\,\cdot\,,\,\cdot\,)$.

\begin{lemma}
\label{l02}
Let $\ms U^{(1)}_n, \dots, \ms U^{(\mf m)}_n$ be disjoint subsets of
$\Xi_n$, $n\ge 1$.  Under assumption \eqref{02},
\begin{equation*}
\limsup_{n\to\infty} \, \theta^{(p)}_n\,
\sum_{j=1}^{\mf m}  D_n(\ms U^{(j)}_n, \sqrt{f^{(j)}_n}) 
\;<\; \infty\;,
\end{equation*}
where $f^{(j)}_n (z) = \mu_n(z) /\hat \pi^{(j)}_n(z)$, $\hat
\pi^{(j)}_n(z) = \pi_n(z)/\pi_n(\ms U^{(j)}_n)$, $z\in \ms U^{(j)}_n$.
\end{lemma}

\begin{proof}
By definition, 
\begin{equation*}
\ms I_n(\mu_n) \;=\; \frac{1}{2} \sum_{x\in \Xi_n}
\sum_{y\in \Xi_n \setminus \{x\}} \pi_n(x) \, R_n(x,y)\,
[\, \sqrt{h(y)} - \sqrt{ h(x)} \,]^2\;,
\end{equation*}
where $h(z) = \mu_n(z) /\pi_n(z)$. The right-hand side is bounded
below by
\begin{equation*}
\frac{1}{2} \sum_{j=1}^{\mf m} 
\sum_{x\in \ms U^{(j)}_n}
\sum_{y\in  \ms  U^{(j)}_n \setminus \{x\}} \pi_n(x) \, R_n(x,y)\,
[\, \sqrt{h(y)} - \sqrt{ h(x)} \,]^2\;.
\end{equation*}
For a fixed $1\le j \le \mf m$, divide and multiply the summand by
$\pi_n(\ms U^{(j)}_n)$ to rewrite this sum as
\begin{equation*}
\sum_{j=1}^{\mf m} D_n(\ms U^{(j)}_n, \sqrt{f^{(j)}_n}) \;.
\end{equation*}
The assertion of the lemma now follows from condition \eqref{02}.
\end{proof}

\begin{lemma}
\label{l15}
Assume that condition {\rm (H4b)} is fulfilled. Fix
$1\le p\le \mf q$. If $\mu_n$ converges weakly to
$\sum_{j\in S_1} \omega_j \, \delta_{x_j}$ for some
$\omega\in \ms P(S_1)$, then, for all $j\in S_1$,
$\displaystyle \lim_{n\to\infty} \mu_n (\ms V^{1,j}_n) \;=\;
\omega_j$.
\end{lemma}

\begin{proof}
As $\mu_n$ converges weakly to
$\sum_{j\in S_1} \omega_j \, \delta_{x_j}$, since the balls
$B_\epsilon(x_j)$ are disjoint,
$\mu_n(B_\epsilon(x_j)) \to \omega_j$ for $j\in S_1$.
Let $\ms A = [\, \cup_{j\in S_1} B_\epsilon(x_j)\,]^c$, so that
$\mu_n(\ms A)\to 0$.

By (H4b), $B_\epsilon(x_j) \cap \Xi_n \subset \ms V^{1,j}_n \subset
[\, B_\epsilon(x_j) \cup  \ms A\,] \cap \Xi_n$ so that
\begin{equation*}
\begin{aligned}
& \omega_j \;=\; \lim_{n\to\infty} \mu_n(B_\epsilon(x_j))
\;\le\; \liminf_{n\to\infty} \mu_n(\ms V^{1,j}_n) \\
&\quad \;\le\; \limsup_{n\to\infty} \mu_n(\ms V^{1,j}_n)
\;\le \; \lim_{n\to\infty} \{ \mu_n(B_\epsilon(x_j) )
\,+\, \mu_n(\ms A)\,\} \;=\; \omega_j\;,
\end{aligned}
\end{equation*}
as claimed.
\end{proof}

\begin{remark}
In the previous lemma, we proved that, for any sequence
$\mu_n\to \sum_{j\in S_1}\omega_j\, \delta_{x_j}$, ${\rm (H4b)}$
implies that $ \mu_n(\ms V_n^{1,j}) \to \omega_j $. The latter is the
\textit{conclusion} of condition {\rm (H6b)}. Note that in this lemma
we do \textit{not} use the additional condition \eqref{02}. Hence,
Lemma \ref{l15} states that {\rm (H4b)} is a stronger hypothesis than
{\rm (H6b)}, but which may be too restrictive in some situations, as
discussed in the paragraph above condition {\rm (H6b)}.
\end{remark}

\begin{proposition}
\label{l03}
Assume that conditions {\rm (H1), (H2), (H4b) and (H5)} hold.
Fix $1\le p\le \mf q$. Let $(\mu_n:n\ge 1)$ be a sequence in
$\ms P(\Xi_n)$ such that
$\mu_n\to \sum_{j\in S_p} \omega_j\, \pi^{(p)}_j$ weakly for some
$\omega\in \ms P(S_p)$, and which fulfills \eqref{02}. 
Then,
\begin{equation*}
\lim_{n\to\infty} \, \theta^{(p)}_n\,
\sum_{x\in \ms V^{p,j}_n} \mu_n (x)\, R^{(p)}_n(x, \ms V^{p,k}_n)
\;=\; \omega_j\, r^{(p)}(j,k) 
\end{equation*}
for all $k\neq j\in S_p$.
\end{proposition}

\begin{proof}
Fix $1\le p\le \mf q$, $k\neq j\in S_p$. Recall from \eqref{10}
the definition of the set $S_{p,j}$ to write
\begin{equation*}
\sum_{x\in \ms V^{p,j}_n} \mu_n (x)\, R^{(p)}_n(x, \ms V^{p,k}_n)
\;=\;
\sum_{i\in S_{p,j}}
\sum_{x\in \ms V^{1,i}_n} \mu_n (x)\, R^{(p)}_n(x, \ms V^{p,k}_n)\;.
\end{equation*}
Let $f_n(x) = \mu_n(x) / \pi^{1,i}_n(x)$,
$G_n(x) = R^{(p)}_n(x, \ms V^{p,k}_n)$, $x\in \ms V^{p,j}_n$. With
this notation, the previous sum is equal to
\begin{equation*}
\sum_{i\in S_{p,j}}
\sum_{x\in \ms V^{1,i}_n} \pi^{1,i}_n(x) \, f_n(x)\, G_n(x)
\;.
\end{equation*}
This expression can be written as
\begin{equation}
\label{11}
\begin{aligned}
& \sum_{i\in S_{p,j}} \mu_n(\ms V^{1,i}_n) \,
\sum_{x\in \ms V^{1,i}_n} \pi^{1,i}_n(x) \, G_n(x) \\
& \quad
\;+\; \sum_{i\in S_{p,j}}
\sum_{x\in \ms V^{1,i}_n} \pi^{1,i}_n(x) \, G_n(x)\,
\Big\{ \, f_n(x) \,-\, \sum_{y\in \ms V^{1,i}_n} \pi^{1,i}_n(y) \,
f_n(y) \,\Big\} \;.
\end{aligned}
\end{equation}

We consider separately the two terms in \eqref{11}. We claim that
\begin{equation}
\label{12}
\lim_{n\to\infty} \theta^{(p)}_n\,
\sum_{i\in S_{p,j}} \mu_n(\ms V^{1,i}_n) \,
\sum_{x\in \ms V^{1,i}_n} \pi^{1,i}_n(x) \, G_n(x) \;=\;
\omega_j \, r^{(p)}(j,k) \;.
\end{equation}
Indeed, rewrite the first term in \eqref{11} as
\begin{equation}
\label{13}
\begin{aligned}
& \sum_{i\in S_{p,j}} \Big\{\, \mu_n(\ms V^{1,i}_n) \,-\,
\omega_j \, \frac{\pi_n(\ms V^{1,i}_n)}{\pi_n(\ms V^{p,j}_n)} \, \Big \} \,
\sum_{x\in \ms V^{1,i}_n} \pi^{1,i}_n(x) \, G_n(x) \\
&\quad +\; \frac{\omega_j}{\pi_n(\ms V^{p,j}_n)} \,
\sum_{i\in S_{p,j}}  
\sum_{x\in \ms V^{1,i}_n} \pi_n(x) \, G_n(x)
\end{aligned}
\end{equation}
By assumption (H1), the second term multiplied by $\theta^{(p)}_n$
converges to $\omega_j \, r^{(p)}(j,k)$. On the other hand, by
\eqref{10}, Lemma \ref{l15} and Corollary \ref{l10b}, the sequences
$\mu_n(\ms V^{1,i}_n)$ and
$\omega_j \pi_n(\ms V^{1,i}_n)/\pi_n(\ms V^{p,j}_n)$ converge to
$\omega_j \, m_{p,j}(i)$ as $n\to\infty$.  Hence,
$\mu_n(\ms V^{1,i}_n) \,-\, \omega_j \, [\, \pi_n(\ms
V^{1,i}_n)/\pi_n(\ms V^{p,j}_n)\,]$ is a bounded sequence which
converges to $0$ as $n\to\infty$.  Moreover,
\begin{equation}
\label{14}
\theta^{(p)}_n\, \sum_{x\in \ms V^{1,i}_n} \pi^{1,i}_n(x) \, G_n(x)
\;\le \; \theta^{(p)}_n\, \frac{\pi_n(\ms V^{p,j}_n)}{\pi_n(\ms V^{1,i}_n)}
\, \frac{1}{\pi_n(\ms V^{p,j}_n)}
\sum_{x\in \ms V^{p,j}_n} \pi_n(x) \, G_n(x)\;.
\end{equation}
Note that the sum on the right-hand side is now carried over
$x\in \ms V^{p,j}_n$, which explains the inequality.  By condition (H1)
and Corollary \ref{l10b}, the right-hand side converges to
$m_{p,j}(i)^{-1} \, r^{(p)}(j,k)$ and is therefore bounded.  The last
two estimates yield that the first term in \eqref{13} multiplied by
$\theta^{(p)}_n$ vanish as $n\to\infty$, which proves claim
\eqref{12}.

We turn to the second term in \eqref{11}. Fix $i\in S_{p,j}$. We claim
that there exists a finite constant $C_0$, independent of $n$, such
that
\begin{equation}
\label{05}
\begin{aligned}
& \Big|\, \theta^{(p)}_n\, \sum_{x\in \ms V^{1,i}_n} \pi^{1,i}_n(x) \, G_n(x)\,
\sum_{y\in \ms V^{1,i}_n} \pi^{1,i}_n(y) \,
\{ \, f_n(x) \,-\,  f_n(y) \,\} \,\Big| \\
& \quad \le\; 
A_n\, \big\{\, \theta^{(p)}_n\,
\sum_{x\in \ms V^{1,i}_n} \mu_n(x) \, G_n(x) \;+\;
C_0 \,\big\} \;
+\;
C_0\, \frac{\beta_n}{A_n}\, D_n(\ms V^{1,i}_n, \sqrt{f_n}) 
\end{aligned}
\end{equation}
for all $A_n>0$, where $\beta_n$ is the sequence introduced in
condition (H5).

Rewrite $f_n(x) - f_n(y)$ as
$[\, \sqrt{f_n(x)} + \sqrt{f_n(y)}\,]\times [\, \sqrt{f_n(x)} -
\sqrt{f_n(y)}\,]$. By Young's inequality $2ab \le a^2 + b^2$ and since
$G_n$ is non-negative and $(\sqrt{b} + \sqrt{a})^2 \le 2 a + 2 b$, the
left-hand side of \eqref{05} is bounded by
\begin{equation}
\label{03}
\begin{aligned}
& \theta^{(p)}_n\, A_n\, \sum_{x,y\in \ms V^{1,i}_n}
\pi^{1,i}_n(x) \, \pi^{1,i}_n(y)
\, G_n(x)\,
\big\{ \, f_n(y) \,+\,  f_n(x) \,\big\} \\
&\quad \;+\;  \frac{\theta^{(p)}_n}{2A_n} 
\sum_{x\in \ms V^{1,i}_n} \pi^{1,i}_n(x) \, G_n(x)\,
\sum_{y\in \ms V^{1,i}_n} \pi^{1,i}_n(y) \, 
\Big\{ \, \sqrt{f_n(y)} \,-\,  \sqrt{f_n(x)} \,\Big\}^2  
\end{aligned}
\end{equation}
for all $A_n>0$. Since $\pi^{1,i}_n(\ms V^{1,i}_n) = 1$,
the first term is bounded by
\begin{equation*}
\theta^{(p)}_n\, A_n\, 
\sum_{x\in \ms V^{1,i}_n} \pi^{1,i}_n(x) \,  f_n(x) \,  G_n(x) \;+\;
\theta^{(p)}_n\, A_n\, \mu_n(\ms V^{1,i}_n) \, 
\sum_{x\in \ms V^{1,i}_n} \pi^{1,i}_n(x) \, G_n(x)\;.
\end{equation*}
We have shown in \eqref{14} that the sequence
$\theta^{(p)}_n\, \sum_{x\in \ms V^{1,i}_n} \pi^{1,i}_n(x) \, G_n(x)$
is bounded.  The previous displayed equation corresponds therefore to
the first term in \eqref{05}.

By assumption (H5), the second term in \eqref{03} is bounded by
\begin{equation*}
\frac{\beta_n}{A_n}\, D_n(\ms V^{1,i}_n , \sqrt{f_n})
\, \theta^{(p)}_n\, \sum_{x\in \ms V^{1,i}_n} \pi^{1,i}_n(x) \, G_n(x)\, \;.
\end{equation*}
Since by \eqref{14} the sequence
$\theta^{(p)}_n \sum_{x\in \ms V^{1,i}_n} \pi^{1,i}_n(x) \, G_n(x)$ is
bounded, this expression corresponds to the second term in \eqref{05}
and completes the proof of this assertion.

We have now all elements to prove the proposition.  Choose a sequence
$A_n \to 0$ such that $\beta_n /A_n \prec \theta^{(p)}_n$. With this
choice, by Lemma \ref{l02},
\begin{equation}
\label{15}
\lim_{n\to\infty} 
\sum_{i\in S_{p,j}} \frac{\beta_n}{A_n}\,
D_n(\ms V^{1,i}_n, \sqrt{f_n}) \;=\;0\;.
\end{equation}

Let $u_n$, $v_n$ be the sequences given by
$u_n = \theta^{(p)}_n\, \sum_{x\in \ms V^{p,j}_n} \mu_n(x) \, G_n(x)$,
$v_n = \theta^{(p)}_n\, \sum_{i\in S_{p,j}} \mu_n(\ms V^{1,i}_n) \,
\sum_{x\in \ms V^{1,i}_n} \pi^{1,i}_n(x) \, G_n(x)$.

By \eqref{11}, \eqref{12}, \eqref{05} and \eqref{15},
$u_n = v_n + w_n$, where $v_n$ converges to $\omega_j\, r^{(p)}(j,k)$
and $|w_n| \le \delta_n u_n + \epsilon_n$ for some sequences
$\delta_n$ and $\epsilon_n$ which vanish as $n\to\infty$. From these
estimates, it is easy to conclude that
$u_n \to \omega_j\, r^{(p)}(j,k)$, as claimed.
\end{proof}

In the proof of the previous proposition, condition (H4b) is only used to apply
Lemma \ref{l15}. We may, therefore, replace this condition by the
conclusion of Lemma \ref{l15}, that is, by condition (H6b).

\begin{lemma}
\label{l12}
Assume that conditions {\rm (H1), (H2), (H5)} and {\rm (H6b)} hold.
Fix $1\le p\le \mf q$. Let $(\mu_n:n\ge 1)$ be a sequence in
$\ms P(\Xi_n)$ such that
$\mu_n\to \sum_{j\in S_p} \omega_j \pi^{(p)}_j$ weakly for some
$\omega\in \ms P(S_p)$, and which fulfills \eqref{02}.  Then,
\begin{equation*}
\lim_{n\to\infty} \, \theta^{(p)}_n\,
\sum_{x\in \ms V^{p,j}_n} \mu_n (x)\, R^{(p)}_n(x, \ms V^{p,k}_n)
\;=\; \omega_j\, r^{(p)}(j,k) 
\end{equation*}
for all $k\neq j\in S_p$.
\end{lemma}

We conclude this section with a crude bound for the sequence $\beta_n$
appearing in condition (H5). This estimate on $\beta_n$ is never used in the sequel, but we report it as it may be useful for future applications.

Turn the set $\ms V^{1,i}_n$ into a graph by introducing the set of
unoriented (because the process is reversible) edges
$\color{blue} \ms E = \{ (x,y) : R_n(x,y)>0\,\}$.  Denote by
$\color{blue} {\rm diam } \ms V^{1,i}_n$ the graph diameter of
$\ms V^{1,i}_n$, i.e., the maximal graph distance
between two points of $\ms V^{1,i}_n$.

\begin{lemma}
\label{l04}
For all $j\in S_p$, $i\in S_{p,j}$, $h\in L^2(\ms V^{1,i}_n)$, 
\begin{equation*}
\max_{x\in \ms V^{1,i}_n} \sum_{y\in \ms V^{1,i}_n} \pi^{1,i}_n(y) \, 
\{ \, h(y) \,-\,  h(x) \,\}^2
\;\le\; 2\, \beta_n\, D_n(\ms V^{1,i}_n , h)
\end{equation*}
where
$\beta_n = {\rm diam } \ms V^{1,i}_n \, \max_{(x,y)\in \ms E}
c^{1,i}_n (x,y)^{-1}$, with
$c^{1,i}_n (x,y) = \pi^{1,i}_n(x)\, R_n(x,y)$.
\end{lemma}

\begin{proof}
Fix $x\in \ms V^{1,i}_n$. For $y\in \ms V^{1,i}_n$, let
$\gamma_y = (x=z_0, \dots, z_m =y)$ be a geodesic from $x$ to $y$. By
Schwarz inequality, and since the length of the path is at most
$ {\rm diam } \ms V^{1,i}_n$,
\begin{equation*}
\begin{aligned}
& \sum_{y\in \ms V^{1,i}_n} \pi^{1,i}_n(y) \,
\{ \, h(y) \,-\,  h(x) \,\}^2 \\
&\quad \;\le\; {\rm diam } \ms V^{1,i}_n 
\sum_{y\in \ms V^{1,i}_n} \pi^{1,i}_n(y) \,
\sum_{(z_i, z_{i+1})\in \gamma_y}
\{ \, h(z_{i+1}) \,-\,  h(z_i) \,\}^2\;.
\end{aligned}
\end{equation*}
Inverting the order of summation, this sum becomes
\begin{equation*}
{\rm diam } \ms V^{1,i}_n  \sum_{(w,z) \in  \ms E}
\big\{ \, h(z) \,-\,  h(w) \,\big\}^2 \sum_{y \in \ms V_n^{1,i}} \pi^{1,i}_n(y) \;.
\end{equation*}
In this equation, the second sum is performed over all
$y\in \ms V^{1,i}_n$ whose path $\gamma_y$ passes through the edge
$(w,z)$. Since the sum over $y$ is bounded by $1$, this expression is
less than or equal to
\begin{equation*}
{\rm diam } \ms V^{1,i}_n\, \max_{(x,y)\in \ms E} \frac{1}{c^{1,i}_n (x,y)}\,
\, \sum_{(w,z) \in  \ms E} \pi^{1,i}_n(z)\, R_n(z,w)\,
\big\{ \, h(z) \,-\,  h(w) \,\big\}^2  \;,
\end{equation*}
as claimed.
\end{proof}

\section{Proof of Theorems \ref{mt1} and \ref{mt2}}
\label{sec5}

The proof of Theorem \ref{mt1} relies on three lemmata.  Recall the
definition of the measures $\pi^{(p)}_j$ introduced in
\eqref{09}. Next result is \cite[Lemma 5.1]{l-mld}. We present its
proof for the sake of completeness.

\begin{lemma}
\label{l39}
Assume that conditions {\rm (H0) -- (H2)} hold.
Fix $0\le p < \mf q$. Then,
\begin{equation*}
\ms I^{(p+1)} (\mu) \;<\; \infty \quad \text{if and only if}\quad
\ms I^{(p)} (\mu) \;=\; 0\;.
\end{equation*}
\end{lemma}

\begin{proof}
Suppose that $\ms I^{(p)} (\mu) = 0$. Then, by condition (H0) for
$p=0$ and \eqref{o-83b} for $1\le p<\mf q$,
$\mu = \sum_{j\in S_p} \omega_j\, \pi^{(p)}_j$ for some
$\omega\in \ms P(S_p)$ and $\bb I^{(p)}(\omega)=0$. By the definition
\eqref{40} of $\bb I^{(p)}$ and \cite[Lemma A.8]{l-mld}, $\omega$ is a
stationary state of the Markov chain $\bb X^{(p)}_t$, that is,
$\omega$ is a convex combination of the measures $M^{(p)}_m$,
$m\in S_{p+1}$ introduced above \eqref{09}:
\begin{equation*}
\omega(j) \;=\; \sum_{m\in S_{p+1}} \vartheta (m)\, M^{(p)}_m(j)\;,
\quad j\,\in\,S_p\;,
\end{equation*}
for some $\vartheta \in \ms P(S_{p+1})$.  Inserting this expression in
the formula for $\mu$ and changing the order of summation yields that
\begin{equation*}
\mu \;=\; \sum_{m\in S_{p+1}} \vartheta (m) \, \sum_{j\in S_p} M^{(p)}_m(j)
\, \pi^{(p)}_j \;=\;
\sum_{m\in S_{p+1}} \vartheta (m) \, \pi^{(p+1)}_m\;,
\end{equation*}
where we used identity \eqref{09} in the last step. This proves the
first assertion of the lemma because
$\bb I^{(p+1)} (\vartheta) <\infty$ for all
$\vartheta\in \ms P(S_{p+1})$. We turn to the converse.

Suppose that $\ms I^{(p+1)} (\mu) < \infty$. In this case, by
\eqref{o-83b},
$\mu = \sum_{m\in S_{p+1}} \vartheta (m) \, \pi^{(p+1)}_m$ for some
$\vartheta \in \ms P(S_{p+1})$. By \eqref{09}, this identity can be
rewritten as
\begin{equation*}
\mu(\,\cdot\,) \;=\; \sum_{j\in S_p} \Big( \, \sum_{m\in S_{p+1}}
\vartheta (m) \, M^{(p)}_m(j)\,\Big)
\, \pi^{(p)}_j (\,\cdot\,) \;.
\end{equation*}
Therefore, by definition of $\ms I^{(p)}$,
$\ms I^{(p)}(\mu) = \bb I^{(p)}(\omega)$, where
$\omega(j) = \sum_{m\in S_{p+1}} \vartheta (m) \, M^{(p)}_m(j)$.  As
the measures $M^{(p)}_m$ are stationary for the chain $\bb X^{(p)}_t$,
so is $\omega$. Thus, by \cite[Lemma A.8]{l-mld},
$\bb I^{(p)}(\omega)=0$, as claimed.
\end{proof}

Next result is the \emph{$\Gamma$-liminf} part of Theorem \ref{mt1} for
measures which can be expressed as convex combinations of the measures
$\pi^{(p)}_j$.

\begin{lemma}
\label{l06}
Assume that conditions {\rm (H1), (H2), (H4b) and (H5)} are in force.
Fix $1\le p\le \mf q$ and a sequence $\mu_n \in \ms P(\Xi_n)$ of
probability measures converging to
$\mu = \sum_{j\in S_p} \omega_j\, \pi^{(p)}_j$ for some
$\omega \in \ms P(S_p)$. Then,
\begin{equation*}
\liminf_{n\to\infty} \theta^{(p)}_n\, \ms I_n(\mu_n)
\;\ge\; \ms I^{(p)}(\mu)\;.   
\end{equation*}
\end{lemma}

\begin{proof}
Consider a sequence $\mu_n \to \mu$. We may assume, without loss of
generality, that \eqref{02} holds.  Fix a function
$\mb h\colon S_p \to (0, \infty)$, and let
$h_n: \ms V^{(p)}_n \to (0,\infty)$ be given by
$h_n = \sum_{j\in S_p} \mb h (j) \, \chi_{\ms V^{p,j}_n}$. Let
$u_n\colon \Xi_n\to \bb R$ be the solution of the Poisson equation
\begin{equation}
\label{07}
\left\{
\begin{aligned}
& \ms L_n u_n \,=\, 0\;, \quad \Xi_n\setminus \ms V^{(p)}_n\;, \\
& u_n \,=\, h_n\;, \quad \ms V^{(p)}_n\;.
\end{aligned}
\right.
\end{equation}
By the standard stochastic representation of the solution of the
Poisson equation $u_n(x)= \mb E_x^{(n)}\big[\,h_n(X_{\tau_n}^{(n)})\,\big]$, with $\tau_n\ge 0$ being the first hitting time of $\ms V_n^{(p)}$, it is clear that $u_n(x) \in (0,\infty)$ for all
$x\in \Xi_n$.

By definition of $\ms I_n$,
\begin{equation*}
\ms I_n(\mu_n) \; \ge \; -\, \int_{\Xi_n}
\frac{\ms L_n u_n}{u_n}\, d\mu_n\;.
\end{equation*}
Since $u_n$ is harmonic on $\Xi_n \setminus \ms V^{(p)}_n$ and
$u_n = h_n$ on $\ms V^{(p)}_n$, by \cite[Lemma A.1]{bgl-lm}, the
right-hand side is equal to
\begin{equation*}
-\, \int_{\ms V^{(p)}_n} \frac{\ms L_n u_n}{u_n}\, d\mu_n
\;=\; -\, \int_{\ms V^{(p)}_n} \frac{\ms L_n u_n}{h_n}\, d\mu_n
\;=\; -\, \int_{\ms V^{(p)}_n}
\frac{(\mf T_{\ms V^{(p)}_n} \ms L_n)\, h_n}{h_n}\, d\mu_n \;,
\end{equation*}
where $\mf T_{\ms V^{(p)}_n} \ms L_n$ has been introduced in
\eqref{eq:2} and is the generator of the trace process
$Y^{\ms V^{(p)}_n}_t$.  Since $h_n$ is constant on each set
$\ms V^{p,j}_n$ (and equal to $\mb h(j)$), the last integral is equal
to
\begin{equation*}
-\, \sum_{j,k\in S_p} \frac{[\, \mb h(k) - \mb h(j)\,]}{\mb h(j)}\,
\sum_{x\in \ms V^{p,j}_n} \mu_n (x)\,
R^{(p)}_n(x, \ms V^{p,k}_n)\;.
\end{equation*}
By Proposition \ref{l03}, this expression multiplied by
$\theta_n^{(p)}$ converges to
\begin{equation*}
-\, \sum_{j\in S_p} \omega_j \, \frac{1}{\mb h(j)}\,
\sum_{k\in S_p} r^{(p)}(j,k) \, [\, \mb h(k) - \mb h(j)\,]
\;=\; -\, \sum_{j\in S_p} \omega_j \, \frac{\bb L^{(p)} \mb h}{\mb h} \;.
\end{equation*}

Summarising, we proved that
\begin{equation*}
\liminf_{n\to\infty} \theta^{(p)}_n\, \ms I_n(\mu_n) \; \ge \;
\sup_{\mb h} \,  -\, \sum_{j\in S_p} \omega_j \, \frac{\bb L^{(p)} \mb
h}{\mb h} \;,
\end{equation*}
where the supremum is carried over all functions
$\mb h: S_p \to (0,\infty)$. By \eqref{40}, \eqref{o-83b}, the
right-hand side is precisely $\ms I^{(p)}(\mu)$, which completes the
proof of the \emph{$\Gamma$-liminf}.
\end{proof}

As in Lemma \ref{l12}, in the proof the previous lemma, condition
(H4b) is only used to apply Lemma \ref{l15}. We may, therefore,
replace this condition by the conclusion of Lemma \ref{l15}, that is,
by condition (H6b). This is the assertion of the next result.

\begin{lemma}
\label{l13}
Assume that conditions {\rm (H1), (H2), (H5)} and {\rm (H6b)} are in
force.  Fix $1\le p\le \mf q$ and a sequence $\mu_n \in \ms P(\Xi_n)$
of probability measures converging to
$\mu = \sum_{j\in S_p} \omega_j\, \pi^{(p)}_j$ for some
$\omega \in \ms P(S_p)$. Then,
\begin{equation*}
\liminf_{n\to\infty} \theta^{(p)}_n\, \ms I_n(\mu_n)
\;\ge\; \ms I^{(p)}(\mu)\;.   
\end{equation*}
\end{lemma}

For a function $u\colon \Xi_n \to (0,\infty)$, denote by $\mf M_{u}\,
\ms L_n$ the generator $\ms L_n$ tilted by the function $u$:
\begin{equation}
\label{21}
{\color{blue} [\,(\mf M_u\, \ms L_n) \, f\, ]\, (x)} \;=\;
\sum_{y\in \Xi_n}
\frac{1}{u(x)} \, R_n(x,y) \, u(y)  \,  \big[\, f(y) \,-\, f(x)
\,\big] 
\end{equation}
for $f\colon \Xi_n \to \bb R$.  Next result is the \emph{$\Gamma$-limsup}
part of Theorem \ref{mt1} for measures which can be expressed as
positive convex combinations of the measures $\pi^{(p)}_j$.

\begin{lemma}
\label{l07}
Assume that conditions {\rm (H1), (H2) and (H4)} hold.  Fix
$1\le p \le \mf q$ and assume that $\ms I^{(p)}$ is a \emph{$\Gamma$-liminf}
for the sequence $\theta^{(p)}_n\, \ms I_n$. Let
$\mu = \sum_{j\in S_p} \omega_j\, \pi^{(p)}_j$
for some $\omega \in \ms P(S_p)$ such that $\omega_j>0$ for all
$j\in S_p$. Then, there exists a sequence $\nu_n \in \ms P(\Xi_n)$
such that $\nu_n \to \mu$ and 
\begin{equation*}
\limsup_{n\to\infty} \theta^{(p)}_n\, \ms I_n(\nu_n) \;\le\; \ms
I^{(p)}(\mu)\;.  
\end{equation*}
\end{lemma}

\begin{proof}
Let $\mu_n$ be the sequence of probability measures given by
\eqref{26}.  By Proposition \ref{l16}, $\mu_n\to \mu$ and
\begin{equation}
\label{30c}
\limsup_{n\to\infty} \theta^{(p)}_n\, \ms I^{(p)}_n(\mu_n) \;\le\; \ms
I^{(p)}(\mu)\;.  
\end{equation}

Since the trace process $Y^{\ms V^{(p)}_n}_t$ is irreducible and
$\mu_n(x)>0$ for all $x\in \ms V^{(p)}_n$, by \cite[Lemma A.3]{l-mld},
there exists $h_n\colon \ms V^{(p)}_n \to (0,\infty)$ such that
\begin{equation*}
\ms I^{(p)}_n(\mu_n) \;=\;  - \,
\int_{\ms V^{(p)}_n} \frac{1}{h_n}\, 
\big [ \, (\, \mf T_{\ms V^{(p)}_n} \ms L_n)\,  h_n \,\big]\,
\; d \mu_n \;.
\end{equation*}
Denote by $u_n$ the harmonic extension of $h_n$ to $\Xi_n$ given by
\eqref{07}. Recall the definition of the tilted generator
$\mf M_{u_n}\, \ms L_n$ introduced in \eqref{21}, and let $\nu_n$ be
its stationary state. By \cite[Proposition C.1]{l-mld},
\begin{equation}
\label{30b}
\ms I_n(\nu_n) \;\le\; \ms I^{(p)}_n (\mu_n)\;.  
\end{equation}
In view of \eqref{30c}, \eqref{30b}, it remains to show that
$\nu_n \to \mu$.

As $\nu_n$ is the stationary state of the Markov chain $X^{(n)}_t$
tilted by $u_n$, by \cite[Proposition 6.3]{bl2}, $\nu_n$ conditioned
to $\ms V^{(p)}_n$ is the stationary state of the Markov chain induced
by the generator $\mf T_{\ms V^{(p)}_n}\, \mf M_{u_n}\, \ms L_n$. By
\cite[Lemma C.4]{l-mld} this generator coincides with
$\mf M_{h_n}\, \mf T_{\ms V^{(p)}_n}\,\ms L_n$. By definition, $\mu_n$
is the stationary state of this later Markov chain. Hence,
$\mu_n (\,\cdot\,) = \nu_n (\,\cdot\,|\, \ms V^{(p)}_n\,)$.

Since $\mu_n \to \mu$ and
$\mu_n (\,\cdot\,) = \nu_n (\,\cdot\,|\, \ms V^{(p)}_n\,)$, it is
enough to show that $\nu_n (\,\ms V^{(p)}_n\,) \to 1$.  Assume, by
contradiction, that $\liminf_n \nu_n(\ms V^{(p)}_n) <1$. Since
$\ms P(\Xi)$ is compact for the weak topology, consider a subsequence,
still denoted by $\nu_n$, such that $\nu_n \to \nu \in \ms P(\Xi)$,
$\lim_n \nu_n(\ms V^{(p)}_n) <1$. By the \emph{$\Gamma$-liminf},
\begin{equation}
\label{17}
\liminf_{n\to\infty} \theta^{(p)}_n\, \ms I_n(\nu_n) \;\ge\; \ms
I^{(p)}(\nu)\;. 
\end{equation}
On the other hand, fix $\epsilon>0$ given by condition (H4b). As
$\nu_n \to \nu$,
\begin{equation*}
\nu \big(\, \cup_{i\in S^\star_p} B_\epsilon(x_i)\, \big)
\; \le\; \liminf_{n\to\infty}
\nu_n \big(\, \cup_{i\in S^\star_p} B_\epsilon(x_i)\, \big)
\; \le\; \liminf_{n\to\infty}
\nu_n \big(\, \ms V^{(p)}_n \, \big) \;<\; 1\;,
\end{equation*}
where $\color{blue} S^\star_p = \cup_{j\in S_p} S_{p,j}$.  Hence,
$\nu (\{ x_i : i\in S^\star_p \}) < 1$, and, by the definition
\eqref{o-83b} of $\ms I^{(p)}$, $\ms I^{(p)}(\nu) = +\infty$. However,
by \eqref{30c}, \eqref{30b},
\begin{equation*}
\limsup_{n\to\infty} \theta^{(p)}_n\, \ms I_n(\nu_n) \;\le\;
\ms I^{(p)}(\mu) \;=\; \bb I^{(p)} (\omega) \;<\; \infty \;,
\end{equation*}
in contradiction with \eqref{17}.

This shows that $\nu_n (\,\ms V^{(p)}_n\,) \to 1$, so that
$\nu_n \to \mu$, and completes the proof of the lemma.
\end{proof}

\begin{proof}[Proof of Theorem \ref{mt1}]
Condition (a) of Definition \ref{def1} follows from assumption
(H1). Condition (c) is proved in Lemma \ref{l39}.  Condition (d)
follows from assumption (H3), the definition of the functional
$\ms I^{(\mf q)}$ and the fact that the functional $\bb I^{(\mf q)}$
vanishes only at the stationary states of the chain
$\bb X^ {(\mf q)}_t$. It remains to examine condition (b) of the
definition.

The proof is by induction in $p$. The case $p=0$ is covered by
condition (H0). Fix $1\le p\le \mf q$ and assume that the result holds
for $0\le p'<p$. 

\smallskip\noindent
\emph{$\Gamma$-liminf}: For all $\mu\in \ms P(\Xi)$
and all sequences of probability measures $\mu_n \in \ms P(\Xi_n)$ such
that $\mu_n\to \mu$,
\begin{equation}
\label{44}
\liminf_{n\to\infty} \theta^{(p)}_n\, \ms I^{(p)}_n(\mu_n) \;\ge\; \ms
I^{(p)}(\mu)\;.  
\end{equation}

Fix a probability measure
$\mu$ on $\Xi$ and a sequence $\mu_n$ converging to $\mu$.  Suppose that
$\ms I^{(p-1)} (\mu)>0$. In this case, since
$\theta^{(p-1)}_n\, \ms I_n$ $\Gamma$-converges to $\ms I^{(p-1)}$ and
$\theta^{(p)}_n/\theta^{(p-1)}_n \to\infty$,
\begin{equation*}
\liminf_{n\to\infty} \theta^{(p)}_n\, \ms I_n(\mu_n) \;=\;
\liminf_{n\to\infty} \frac{\theta^{(p)}_n}{\theta^{(p-1)}_n}
\, \theta^{(p-1)}_n\, \ms I_n(\mu_n) \;\ge\;
\ms I^{(p-1)} (\mu)\,
\lim_{n\to\infty} \frac{\theta^{(p)}_n}{\theta^{(p-1)}_n}
\;=\; \infty\;.
\end{equation*}
On the other hand, by Lemma \ref{l39}, $\ms I^{(p)} (\mu) =
\infty$. This proves the \emph{$\Gamma$-liminf} convergence for measures
$\mu$ such that $\ms I^{(p-1)} (\mu)>0$.

Assume that $\ms I^{(p-1)} (\mu)=0$. By Lemma \ref{l39}, there exists
a probability measure $\omega$ on $S_p$ such that
$\mu = \sum_{j\in S_p} \omega_j\, \pi^{(p)}_j$. Hence, assertion
\eqref{44} follows from Lemma \ref{l06}.

\smallskip\noindent
\emph{$\Gamma$-limsup}.
Fix $\mu\in \ms P(\Xi)$.
If $\ms I^{(p)}(\mu) = \infty$, there is nothing to prove. Assume,
therefore, that $\ms I^{(p)}(\mu) < \infty$.  Hence, by \eqref{o-83b},
$\mu = \sum_{j\in S_p} \omega_j \, \pi^{(p)}_j$ for some
$\omega \in \ms P(S_p)$.

By Lemmata B.4 and B.3 in \cite{l-mld}, it is enough to prove the
\emph{$\Gamma$-limsup} for measures
$\mu = \sum_{j\in S_p} \omega_j \, \pi^{(p)}_j$ for some
$\omega \in \ms P(S_p)$ such that $\omega_j>0$ for all $j\in S_p$.
This is the content of Lemma \ref{l07}.
\end{proof}

\begin{remark}
\label{rm3b}
The main ingredients of the proof of the \emph{$\Gamma$-liminf} are the
harmonic identity stated in \cite[Lemma A.3]{l-mld}, and Proposition
\ref{l03} which extends condition (H1) to measures converging to a
convex combination of the measures $\pi^{(p)}_j$ and satisfying
\eqref{02}.
\end{remark}

\begin{remark}
\label{rm3}
The proof of the \emph{$\Gamma$-limsup} forced the wells $\ms V^{1,j}_n$ to
contain macroscopic balls $B_\epsilon(x_j)$. Indeed, in the proof that
$\nu_n(\ms V^{(p)}_n) \to 1$ in Lemma \ref{l07}, we assumed by
contradiction that $\nu_n \to \nu$ and
$\lim_n \nu_n(\ms V^{(p)}_n) <1$. To conclude that
$\nu(\{x_j : j\in S^\star_{p}\})<1$ we had to suppose (since we
adopted the weak topology in the space of measures) that the sets
$\ms V^{1,j}_n$ contain open balls. It is in this proof, and only
here, that we needed the space $\ms P(\Xi)$ to be compact.
\end{remark}

\begin{lemma}
\label{l14}
Assume that conditions {\rm (H1), (H2), (H4a), (H6a)} hold.  Fix
$1\le p \le \mf q$ and $\mu = \sum_{j\in S_p} \omega_j\, \pi^{(p)}_j$
for some $\omega \in \ms P(S_p)$ such that $\omega_j>0$ for all
$j\in S_p$. Then, there exists a sequence $\nu_n \in \ms P(\Xi_n)$
such that $\nu_n \to \mu$ and 
\begin{equation*}
\limsup_{n\to\infty} \theta^{(p)}_n\, \ms I_n(\nu_n) \;\le\; \ms
I^{(p)}(\mu)\;.  
\end{equation*}
\end{lemma}

\begin{proof}
Let $\mu_n$ be the sequence of probability measures given by
\eqref{26}.  By Proposition \ref{l16}, $\mu_n\to \mu$ and
\begin{equation}
\label{30d}
\limsup_{n\to\infty} \theta^{(p)}_n\, \ms I^{(p)}_n(\mu_n) \;\le\; \ms
I^{(p)}(\mu)\;.  
\end{equation}
Let $f_n \colon \ms V^{(p)}_n \to \bb R_+$ be the function given by
$f_n (x) = \mu_n(x) / \pi^{(p)}_n (x)$, where recall
$\pi^{(p)}_n(x) = \pi_n(x)/\pi_n(\ms V^{(p)}_n)$. An elementary
computation yields that
$f_n(x) = \omega_j \pi_n(\ms V^{(p)}_n)/\pi_n(\ms V^{p,j}_n)$ for
$x\in \ms V^{p,j}_n$.

We extend $f_n$ to $\Xi_n$ solving the Poisson equation \eqref{07}
with $h_n = \sqrt{f_n}$. Denote by $\sqrt{g_n}:=u_n$ the solution. Let
$\nu_n = Z_n^{-1} \, g_n \, \pi_n$, where $Z_n$ is a
normalising constant which turns $\nu_n$ into a probability measure.

We claim that $Z_n\to 1$ and $\nu_n \to \mu$. By definition,
\begin{equation*}
Z_n \;=\; \sum_{x\not\in \ms V^{(p)}_n} g_n(x) \, \pi_n(x)
\;+\;
\sum_{x\in \ms V^{(p)}_n} f_n(x) \, \pi_n(x)\;.
\end{equation*}
By condition (H6a), the function $f_n$ is bounded. As $\sqrt{g_n}$ is
the solution of the Poisson equation, $g_n$ is bounded by the same
constant. Hence, by Lemma \ref{l17}, the first term, which is bounded
by $C_0\, \pi_n([\ms V^{(p)}_n]^c)$ for some finite constant $C_0$,
vanishes as $n\to\infty$. The second one, by definition of $f_n$, is
equal to $\pi_n(\ms V^{(p)}_n)$ which, by Lemma \ref{l17}, converges to
$1$ as $n\to\infty$. This proves that $Z_n \to 1$.

By similar reasons, $\nu_n([\ms V^{(p)}_n]^c) \to 0$. On the other
hand, for any subset $\ms A$ of $\ms V^{(p)}_n$,
$\nu_n (\ms A) = Z_n^{-1} \,\pi_n(\ms V^{(p)}_n)\, \mu_n(\ms A)$. As
$Z_n \to 1$, $\pi_n(\ms V^{(p)}_n)\to 1$, and, by Proposition
\ref{l16}, $\mu_n\to \mu$, we may conclude that $\nu_n\to \mu$, as
claimed.

We turn to the proof that
$\limsup_n \theta^{(p)}_n\, \ms I_n(\nu_n) \le \ms I^{(p)} (\mu)$. By
\eqref{f6} and the definition of $\nu_n$,
\begin{equation*}
\ms I_n(\nu_n) \;=\; Z_n^{-1} \,
\< \, \sqrt{g_n} \, , \, (-\ms L_n)\, \sqrt{g_n} \,\>_{\pi_n} \;.
\end{equation*}
By \cite[Corollary A.2]{bgl-lm} and the definition of $h_n$, the
right-hand side is equal to
\begin{equation*}
Z_n^{-1} \, \pi_n(\ms V^{(p)}_n)\,
\< \sqrt{f_n} \, , (- \mf T_{\ms V^{(p)}_n}
\ms L_n) \sqrt{f_n}\>_{\pi^{(p)}_n} \;=\;
Z_n^{-1} \, \pi_n(\ms V^{(p)}_n)\,
\ms I^{(p)}_n(\mu_n) \;.
\end{equation*}
To complete the proof, it remains to recall \eqref{30d} and that
$Z_n \to 1$, $\pi_n(\ms V^{(p)}_n)\to 1$.
\end{proof}

\begin{proof}[Proof of Theorem \ref{mt2}]
The proofs of conditions (a), (c) and (d) of Definition \ref{def1} and
the one of the \emph{$\Gamma$-liminf} are analogous to the one presented in
Theorem \ref{mt1}. For the later, we just apply Lemma \ref{l13}
instead of Lemma \ref{l06}. The proof of the \emph{$\Gamma$-limsup} is
identical.
\end{proof}

\section{Random walks in potential fields}
\label{sec6}

In this section, based on the theory developed in Section \ref{sec1},
we derive the full $\Gamma$-expansion of the large deviations rate
functional of a random walk in a potential field.  Let
$\color{blue}\Xi= \bb T^d$ with $d\in \bb N$,
$\color{blue} e_1,\ldots, e_d$ denote the vectors of the canonical
basis of $\bb R^d$, and fix a potential function
$\color{blue}F\colon \Xi\to \bb R$ satisfying the following
assumptions:

\begin{enumerate}

\item[\textbf{(F1)}]\label{a:h1} $F\in \ms C^{2,1}(\Xi)$, i.e.,
$F\in \ms C^2(\Xi)$ and all second-order partial derivatives of $F$
are Lipschitz continuous.

\item[\textbf{(F2)}]\label{a:h0} $F$ has finitely-many critical
points $\color{blue}\mc P:=\{z_1,\ldots, z_k\}$.

\item[\textbf{(F3})]\label{a:h2} All the eigenvalues of the Hessian of
$F$ at the critical points are non-zero.

\item[\textbf{(F4)}]\label{a:h3} Let
${\color{blue}\mc S}\subset \mc P$ be the set of saddle points of $F$.
The Hessian of $F$ at saddle points $z\in \mc S$ has only one negative
eigenvalue.

\item[\textbf{(F5)}]\label{a:h4} All saddle points $z\in \mc S$ are at
the same height: there exists $h\in \bb R$ such that
$F(z)={\color{blue}h}$ for all $z\in \mc S$.
	
\item[\textbf{(F6)}]
\label{a:h5}
Let $\color{blue}\mc M = \{ m_1,\ldots, m_{\mf n}\}$ be the set of
local minima of $F$.  The set $\{x\in \Xi: F(x)\leq h\}$ is connected
and it contains $\mc M$.  Each connected component of
$\{x\in \Xi: F(x)<h\}$ contains a unique critical point of $F$ (one of
the local minima). Denote by $\color{blue} \ms C_k$ the connected
component which contains $m_k$.

\end{enumerate}

\begin{remark}
The asymptotic behaviour of the mean jump rates \eqref{20} for random walks in
potential fields stated in condition (H1) has been studied in
\cite{lmt2015} under condition (F5). It should be possible to remove
this condition, but this has not been done yet and there are serious
technical problems to compute some capacities without this assumption.
\end{remark}

Let $\Xi_n$ denote the discretisation of $\Xi$, and $\pi_n$ the
probability measure on $\Xi_n$ defined by
\begin{equation*}
{\color{blue}\pi_n(x)} \;:=\; \frac 1{Z_{F,n}}\, e^{-n F(x)}\;,\quad x\in
\Xi_n\;, 
\end{equation*}
where $Z_{F,n}$ is the partition function
$\color{blue}Z_{F,n} := \sum_{x\in \Xi_n} \exp\left(-n F(x)\right)$.  Let
$X^{(n)}_t$ be the continuous-time Markov chain on $\Xi_n$ with
transition rates $R_n(x,y)$ given by
\begin{equation*}
{\color{blue}
R_n(x,y)}\::=\:
\begin{cases}
e^{-\frac{n}{2}\left(F(y)-F(x)\right)} &\text{if}\ |y-x|=1/n\\
0 &\text{otherwise}\ ;
\end{cases}
\end{equation*}
the corresponding generator $\ms L_n$ and level two large deviations
rate functional $\ms I_n$ are then defined as in \eqref{47} and
\eqref{f4}, respectively. With these definitions, the Markov chain
$X^{(n)}_t$ is reversible with respect to $\pi_n$; hence, $\ms I_n$
reads as in \eqref{f6}.

\subsection{Metastable behaviour of $X^{(n)}_t$}
\label{ss6.1}

The rate functionals appearing in the $\Gamma$-expansion of $\ms I_n$
are expressed in terms of finite state Markov chains which describe
the evolution of $X^{(n)}_t$ in certain time-scales. This is the
content of this subsection, whose results are taken from
\cite{lmt2015} and included in sake of completeness.

\smallskip\noindent{\bf Valleys.}  Recall that
$\mc M = \{{ m_1},\ldots,{m_{\mf n}}\}$ stands for the set of local
minima of $F$, and set ${\color{blue} S_1}=\{1,\ldots,\mf n\}$.  By
condition (F5), $h$ represents the common height of the saddle points:
$F(z) = h$ for all $z\in \mc S$.  Fix $\varepsilon>0$, such that
$\max_{k\in S_1}F(m_k) + \varepsilon<h$ and, for $k\in S_1$, let
${\color{blue}\mc W(m_k)}\subseteq \Xi$ be the connected component of
$\{x\in\Xi: F(x)<F(m_k)+\varepsilon\}$ that contains $m_k$. Note that
we have chosen $\varepsilon>0$ small enough for $m_k$ to be the unique
critical point of $F$ in $\mc W(m_k)$, $k\in S_1$.

Let
\begin{equation*}
{\color{blue}\mc W^{m_k}_n} :=\mc W(m_k) \cap\Xi_n \;, \quad
k\in S_1 \;,
\end{equation*}
and let ${\color{blue}\hat d_k}=h-F(m_k)$ be the depth of the well
$\mc W^{m_k}_n$. Denote by $\color{blue} d_1<d_2<\cdots<d_{\mf q }$
the depths in increasing order so that
$\{ d_1, d_2, \cdots , d_{\mf q } \} = \{ \hat d_k : k\in S_1\}$. The
cardinality of this set defines the number $\mf q$ of time-scales at
which the Markov chain $X^{(n)}_t$ exhibits a metastable behaviour.

Let
\begin{equation*}
{\color{blue} J_p} \;:=\; \{k\in S_1: \hat d_k \geq d_p\}
\;=\; \big\{\, k\in S_1: h - F(m_k)  \geq d_p \,\big\} \;,
\end{equation*}
$1\leq p \leq \mf q$, be the indices of the wells $\mc W^{m_k}_n $
with depth larger than or equal to $d_p$. Hence,
$J_1\supsetneq J_2 \supsetneq \cdots \supsetneq J_\mf q$, $J_1 = S_1$,
and $\{m_k: k\in J_{\mf q}\}$ coincides with the set of the global
minima of $F$.

Let $\mc M_p$ be the set of local minima of $F$ in wells with depth
larger than or equal to $d_p$:
$\color{blue} \mc M_p = \{m_k : k\in J_p\}$. Enumerate this set to
write it as $\mc M_p = \{m_{p,1}, \dots, m_{p,\mf n_p}\}$, where
$\color{blue} \mf n_p = |J_p|$ stands for the number of elements of
$J_p$. Note that $\mf n_1 = \mf n$ is the number of local minima of
$F$ and $\mf n_{\mf q}$ the number of global minima. Moreover, since
$\mc M_p\subset \mc M_{q}$, $1\le q < p\le \mf q$, for each
$1\le j\le \mf n_p$, there exists $1\le k\le \mf n_{q}$ such that
$m_{p,j} = m_{q,k} $.

We are finally in a position to introduce the valleys. Fix
$1\le p\le \mf q$, and let
$\color{blue} \ms V^{p,j}_n = \mc W^{m_{p,j}}_n$,
$j\in {\color{blue} S_p = \{1, \dots, \mf n_p\}}$.

\smallskip\noindent{\bf A graph.}  Let $\bb G = (S_1, E)$ be a
weighted graph whose vertices are the elements of $S_1$. To define the
set of edges, let $\color{blue} \overline{\ms C}_k$ be the closure of
the connected component $\ms C_k$ introduced in condition (F6).  Place
an edge between $i$ and $j\in S_1$ if, and only if, there exists a
saddle point $z\in \mc S$ belonging to
$\overline{\ms C}_i \cap \overline{\ms C}_j$. In this case, the weight
of the edge between $i$ and $j$, denoted by $\bs c(i,j)$, is set to be
\begin{equation}
\label{v22}
{\color{blue}\bs c(i,j)} \;=\;
\sum_{z\in \mc S \cap (\overline{\ms C}_i \cap \overline{\ms C}_j) } 
\,  \frac{\gamma (z)}{2\pi\sqrt{- \det {\rm Hess}\, F(z)}} \;,    
\end{equation}
where $\color{blue} - \, \gamma(z)$ is the unique negative eigenvalue of
the Hessian of $F$ at $z$. The graph $\bb G$ has to be interpreted as
an electrical network, where the weights $\bs c(a,b)$ represent the
conductances.

Denote by $\{Y_k : k\ge 0\}$ the discrete-time random walk on $S_1$
which jumps from $i$ to $j$ with probability
\begin{equation}
\label{v42}
p(i,j) \;=\; \frac{\bs c(i,j)}{\sum_{k\in S_1} \bs c (i,k)}\;\cdot
\end{equation}
By condition (F6), the Markov chain $Y_k$ is irreducible.  Denote by
$\color{blue} \bb P^Y_i$, $i\in S_1$, the distribution of the chain
$Y_k$ starting from $i$ and by $h_{A,B}$, $A$, $B\subset S$,
$A\cap B=\varnothing$, the equilibrium potential between $A$ and $B$:
\begin{equation*}
{\color{blue} h_{A,B} (i)}
\;=\; \bb P^Y_i \big[ H_A < H_B \big]\;,\;\; i\in S_1\;,
\end{equation*}
where $H_C$, $C\subset S_1$, represents the hitting time of $C$
introduced in \eqref{o-201}. The conductance between $A$ and $B$ is
defined as
\begin{equation*}
{\color{blue} \Cap_{\bb G} (A,B)}
\;:=\; \frac 12 \sum_{i,j\in S_1} \bs c (i,j)
\big[h_{A,B} (j)  - h_{A,B} (i) \big]^2 \;.
\end{equation*}

Fix $1\le p\le \mf q$, $k\in S_p$, and recall that this index
corresponds to a local minimum $m_{p,k}$ of $F$. Let
$\mf b_{p,k}\in S_1$ be the index of the local minimum $m_{p,k}$ when
regarded as an element of $S_1$ so that $m_{p,k} = m_{1,\mf
b_{p,k}}$. For $i$, $j$ in $S_p$, let
\begin{equation}
\label{v41}
\begin{aligned}
{\color{blue}\bs c_p (i,j)} \;:=\; \frac 12 \Big\{ \,
& \Cap_{\bb G}( \{\mf b_{p,i}\}, S_p\setminus
\{\mf b_{p,i}\}) \;+\;
\Cap_{\bb G}( \{\mf b_{p,j}\}, S_p\setminus
\{\mf b_{p,j}\})  \\
-\; & \Cap_{\bb G}( \{\mf b_{p,i} , \mf b_{p,j}\},
S_p\setminus \{\mf b_{p,i} , \mf b_{p,j} \}) \, \Big\} \;.
\end{aligned}
\end{equation}
Clearly, $\bs c_1 (i,j) = \bs c (i,j)$ for $i$, $j\in S_1$. Let
$\mc M_{\mf q+1} = \varnothing$, and
\begin{equation}
\label{v23}
{\color{blue}\bs r_p (i,j)} \;:=\;
\begin{cases}
\bs c_p (i,j)/\gamma (m_{p, i}) & m_{p, i} \in
\mc M_{p}  \setminus \mc M_{p+1} \,,\, j\in S_p \;, \\
0 & m_{p,i} \in \mc M_{p+1} \,,\, j\in S_p\;,
\end{cases}
\end{equation}
where, for $m\in \mc M$,
\begin{equation*}
{\color{blue}\gamma (m)} \;:=\; 
\frac{1}{\sqrt{\det {\rm Hess}\, F(m)}} \;.
\end{equation*}
Denote by $\color{blue} \bb X^{(p)}_t$ the $S_p$-valued Markov chain
with jump rates $\bs r_p$.

Note that the Markov chain $\bb X^{(p)}_t$ has a very simple
structure. Suppose that $p<\mf q$. A point $k\in S_p$ is an absorbing
point if $m_{p,k} \in \mc M_{p+1}$, otherwise it is a transient
point. If $p=\mf q$, the chain is irreducible.

\smallskip\noindent{\bf Model reduction.}  Let
$\color{blue} \ms V^{(p)}_n = \cup_{j\in S_p} \ms V^{p,j}_n$,
$1\le p\le \mf q$, and $\Psi^p_n \colon \ms V^{(p)}_n \to S_p$ the
projection given by
\begin{equation}
\label{v25}
{\color{blue}\Psi^p_n} \;=\; \sum_{j\in S_p} j \, \chi_{_{\ms V^{p,j}_n}}  \;.
\end{equation}
Denote by $\color{blue} Y^{p,n} (t)$ the trace of the Markov chain
$X^{(n)}_t$ on the set $\ms V^{(p)}_n $, as defined in Section
\ref{sec1}. Let $\bb X^{p,n}_t$ be the projection of the Markov chain
$Y^{p,n} (t)$ speeded-up by $\color{blue}\theta^{(p)}_n:=n \exp\{n d_p\}$:
\begin{equation*}
{\color{blue}  \bb X^{p,n}_t} \;:=\;
\Psi^p_n \big(\, Y^{p,n}  (t \theta^{(p)}_n) \, \big)\;,
\end{equation*}
By \cite[Theorem 2.4]{lmt2015}, the process $\bb X^{p,n}_t$ converges
in the Skorohod topology to the Markov chain
$\color{blue} \bb X^{(p)}_t$ introduced above.

\subsection{$\Gamma$-expansion of $\ms I_n$}

The full $\Gamma$-expansion of $\ms I_n$ takes the form
\eqref{f05-general}, cf.\ \eqref{eq:expansion-RWPF}.  The definition
of the functionals which arise in the $\Gamma$-expansion of $\ms I_n$
requires some notation. Let $\ms J: \ms P(\Xi)\to [0,+\infty)$ be
given by
\begin{equation*}
{\color{blue} \ms J(\mu)}\::=\:\int_\Xi \ms G\, d\mu\ ,
\end{equation*}
where $\ms  G = \ms G_F \colon \Xi \to \bb R_+$ is defined as
\begin{equation}
\label{eq:G-cosh}
{\color{blue} \ms G(x)} \::=\:
\sum_{i=1}^d 2\left\{\cosh\left(\frac{1}{2}\partial_i F(x)\right)-1\right\}\ .
\end{equation}
Lemma \ref{lem:zero-expansion} states that $\ms I_n$
$\Gamma$-converges to $\ms J$. This is the first step in the
expansion. 

We turn to the second term in the expansion. Let
$\ms J^{(0)} \colon\ms P(\Xi)\to [0,+\infty]$ be given by
\begin{equation}
\label{eq:I0}
{\color{blue}
\ms J^{(0)}(\mu)}\::=\: \int_\Xi \zeta\, d\mu\ ,
\end{equation}
where
\begin{equation}
\label{eq:zeta}
{\color{blue}
\zeta(x)}:= \begin{cases}
\sum_{i=1}^d \max\{\:-\,\xi_i(x) \,,\, 0\:\} &\text{if}\ x \in \mc P\\
+\infty &\text{otherwise}\ ,
\end{cases}
\end{equation}
and $\color{blue}\xi_i(x)$, $i=1,\ldots,d$, denote the eigenvalues of
${\rm Hess }\,  F (x)$. Proposition \ref{pr:step-1-RW} asserts that 
$n\, \ms I_n$ $\Gamma$-converges to $\ms J^{(0)}$. 

Finally, let
$\color{blue} \bb J^{(p)} \colon \ms P (S_p) \to [0,+\infty)$,
$1\le p\le \mf q$, be the rate functional given by \eqref{40}, where
$\bb L_p$ corresponds to the generator of the Markov chain
$\bb X^{(p)}_t$ introduced in the previous subsection, and
$\color{blue} \ms J^{(p)} \colon \ms P(\Xi) \to [0,+\infty]$ the one
given by \eqref{o-83b}.  The main result of this section reads as
follows

\begin{theorem}
\label{mt3}
Assume that conditions {\rm (F1) -- (F6)} are in force. Then, the full
$\Gamma$-expansion, as specified in Definition \ref{def1}, of the
level two large deviations rate functional $\ms I_n$ of the random
walk introduced at the beginning of this section is given by
\begin{equation}\label{eq:expansion-RWPF}
\ms I_n \;=\; \ms J \;+\; \frac{1}{n} \, \ms J^{(0)}  \;+\;
\sum_{p=1}^{\mf q} \frac{1}{n\, e^{d_p n}} \, \ms J^{(p)}  \;.
\end{equation}
\end{theorem}

\subsection{$\Gamma$-convergence of $\ms I_n$}
\label{sec6.1}

In this subsection, we prove the first step in the expansion, as
stated in the next result.

\begin{lemma}
\label{lem:zero-expansion}
Assume that conditions {\rm (F1)} and {\rm (F2)} are in force. Then,
$\ms I_n$ $\Gamma$-converges to $\ms J$.
\end{lemma}

\begin{proof}
We prove the \emph{$\Gamma$-liminf} and \emph{$\Gamma$-limsup} inequalities
separately.
	
\smallskip\noindent
\emph{$\Gamma$-liminf}. 
Fix $\mu \in \ms P(\Xi)$,
and let $\mu_n \in \ms P(\Xi_n)$ be a sequence such that
$\mu_n \to\mu$.  For $u_n:= e^{\frac{n}{2}F}$, we have, for all
$x\in \Xi_n$,
\begin{equation*}
-\frac{\ms L_n u_n(x)}{u_n(x)}
\;=\; \sum_{i=1}^d \left\{e^{-\frac{n}{2}\left(F(x+\frac{e_i}{n})-F(x)
  \right)}
+e^{-\frac{n}{2}\left(F(x-\frac{e_i}{n})-F(x)\right)}-2	\right\}\ .
\end{equation*}
By the smoothness of $F$ in (F1), recalling the definition of $\ms G$
in \eqref{eq:G-cosh}, we get
\begin{equation*}
\sup_{x\in \Xi_n}\Big|\, \frac{-\ms L_n u_n(x)}{u_n(x)}-\ms G(x)
\, \Big | \;=\; o(1)\ .
\end{equation*}
Hence, plugging such a $u_n$ into the variational formulation of
$\ms I_n$ in \eqref{f4}, we obtain
\begin{equation*}
\label{eq:liminf-step0}
\liminf_{n\to \infty} \ms I_n(\mu_n)
\ge \liminf_{n\to \infty} \sum_{x\in \Xi_n}\mu_n(x)\, \ms G(x)= \ms J(\mu)\ ,
\end{equation*}
where the second step follows by $\mu_n\to \mu$ and
$\ms G\in \ms C(\Xi)$. This concludes the proof of the
\emph{$\Gamma$-liminf}.
	
\smallskip\noindent
\emph{$\Gamma$-limsup}.
It suffices to prove
the $\limsup$ inequality only for Dirac measures. Indeed, the set of
finite convex combinations of Dirac masses is dense in $\ms P(\Xi)$,
as well as $\ms I_n$ is convex and $\ms J$ is linear and of the form
$\ms J(\mu)=\mu\, \ms G$ with $\ms G\in \ms C(\Xi)$.

Therefore, fix $x \in \Xi$, let $\mu=\delta_x$, and consider a
quadratic-like potential centered at $x$ satisfying the following
properties: ${\color{blue}V}\in \ms C^2(\Xi)$ and such that
	\begin{equation}\label{eq:H-lb}
		V(y)\ge 
		\frac{\Vert y-x\Vert^2}{2} \wedge 1\, ,\qquad y \in \Xi\, .
	\end{equation}
	As a consequence of these requirements, we have $V(y)=0$ iff $y=x$, and $\nabla V(x)=0$. 
	Further, it is not difficult to check that
	\begin{align}\label{eq:weak-conv-zero}
		\mu_n:= \frac{e^{-n V}}{\sum e^{-n V}} \xrightarrow{n\to \infty} \mu=\delta_x\, . 
	\end{align}  Indeed, 
	for every $\phi\in \ms C(\Xi)$ and $\delta \in (0,1)$, 
	\begin{align*}
		\left|\langle \mu_n,\phi\rangle -\langle \delta_x,\phi\rangle \right|&\le \sup_{\substack{y\in \Xi\\
				\Vert
          y-x\Vert<\delta}}\left|\phi(y)-\phi(x)\right|+
          2\|\phi\|_\infty  \frac{\sum_{\Vert y- x\Vert>\delta}e^{-n
          V(y)}}{\sum_{y\in \Xi_n}e^{-n V(y)}}\;,
\end{align*}
which, by the continuity of $\phi$ at $x\in \Xi$ and \eqref{eq:H-lb},
vanishes taking first $n\to \infty$, and then $\delta \to 0$. As for
the convergence of the rate functionals along $(\mu_n)_n$, \eqref{f6} and
a simple computation yield
	\begin{align*}
		\ms I_n(\mu_n)  &= \sum_{y\in \Xi_n} \mu_n(y) \sum_{i=1}^d \left\{ e^{-\frac{n}{2}\left(F(y+\frac{e_i}{n})-F(y) \right)}+e^{-\frac{n}{2}\left(F(y-\frac{e_i}{n})-F(y)\right)}-2 \right\}\\
		&\qquad-\sum_{y\in \Xi_n}\mu_n(y)\sum_{i=1}^d\left\{e^{-\frac{n}{2}\left(V(y+\frac{e_i}{n})-V(y)\right)}+e^{-\frac{n}{2}\left(V(y-\frac{e_i}{n})-V(y) \right)}-2 \right\}\\
		&= \sum_{y\in \Xi_n}\mu_n(y)\sum_{i=1}^d 2\left\{\cosh\left(\frac{1}{2}\partial_i F(y)\right)-1\right\}\\
		&\qquad- \sum_{y\in \Xi_n} \mu_n(y)\sum_{i=1}^d 2\left\{\cosh\left(\frac{1}{2}\partial_i V(y)\right)-1  \right\}+ O\left(\frac{1}{n} \right)\ .\end{align*}
	Since $\partial_iF, \partial_iV\in \ms C(\Xi)$, $i=1,\ldots, d$, and $\mu_n\to \mu=\delta_x$ (see \eqref{eq:weak-conv-zero}), 
	\begin{align*}
		\limsup_{n\to \infty}\ms I_n(\mu_n)=  \ms G(x)-\sum_{i=1}^d  2\left\{\cosh\left(\frac{1}{2}\partial_i V(x)\right)-1\right\} = \ms G(x)\, ,
	\end{align*}
	where in the last step we used that $\partial_i V(x)=0$ for all $i=1,\ldots, d$. 	This concludes the proof of the lemma.
\end{proof}

\subsection{$\Gamma$-convergence of $n\, \ms I_n$}
\label{sec6.1}

We turn to the second term of the expansion stated in the next result.

\begin{proposition}
\label{pr:step-1-RW}
Assume that conditions {\rm (F1) -- (F3)} are in force. Then,
$n\, \ms I_n$ $\Gamma$-converges to $\ms J^{(0)}$ given in
\eqref{eq:I0}. 
\end{proposition}

In what follows, for all $x \in \Xi$ and $\varepsilon \in (0,1)$, let
$\color{blue}Q_\varepsilon(x)$ denote the $d$-dimensional cube of size
$\varepsilon$ centered at $x$. Moreover, define the discretizations of
such cubes
${\color{blue}Q_\varepsilon^n(x)}:= Q_\varepsilon(x)\cap \Xi_n$, and
their discrete inner boundary:
\begin{equation*}
	{\color{blue}
	\partial Q_\varepsilon^n(x)}:=\left\{y\in Q_\varepsilon^n(x): \text{there exists}\ y'\in \Xi_n\setminus Q_\varepsilon^n(x)\ \text{such that}\ |y-y'|=1/n \right\}\ .	
\end{equation*}
Finally, set
${\color{blue}\mathring
Q^n_{\varepsilon}(x)}:=Q^n_{\varepsilon_n}(x)\setminus \partial
Q^n_{\varepsilon_n}(x)$.
 
\begin{lemma}
\label{lem:finite-energy}
Assume that conditions {\rm (F1) -- (F3)} are in force, and that
$\mu_n\in \ms P(\Xi_n)$ satisfies
\begin{equation}\label{eq:finite-energy}
\liminf_{n\to\infty} \, n\, \ms I_n(\mu_n) \;<\; \infty\;.
\end{equation}
Then, for every $\varepsilon \in (0,1)$ small enough, we have
\begin{align}
\label{eq:claim1}
\limsup_{n\to \infty} n \, \mu_n\Big(\Xi_n\setminus \bigcup_{z\in \ms
P}Q^n_\varepsilon(z)\Big)<\infty\ ,
\end{align}
and
\begin{align}
\label{eq:claim2}
\limsup_{n\to \infty} n\sum_{z\in \ms P}\sum_{x\in
Q^n_\varepsilon(z)}\mu_n(x)\, |x-z|^2<\infty\ .
\end{align}

\begin{proof}
Letting $\color{blue}\sum_{xy}$ stand for the summation over unordered
pairs of nearest-neighbours $x, y$ of $\Xi_n$, as well as
${\color{blue}\nabla^n_{xy}F}:=n\left(F(y)-F(x)\right)$ and
${\color{blue}\nabla^n_i
F(x)}:=n\left(F(x+\frac{e_i}{n})-F(x)\right)$, by \eqref{f6},
expanding the square,
\begin{align*}
n\, \ms I_n(\mu_n)&= n\sum_{xy}e^{-\frac{n}{2}\left(F(x)+F(y)\right)}
\left(\sqrt{\mu_n(x)}\,e^{\frac{n}{2}F(x)}-\sqrt{\mu_n(y)}\,
e^{\frac{n}{2}F(y)}\right)^2\\ 
&=n\sum_{x\in \Xi_n}\mu_n(x)\sum_{i=1}^d
\Big(e^{-\frac{1}{2}\nabla^n_i F(x)}+e^{\frac{1}{2}\nabla^n_i
F(x-\frac{e_i}{n})}\Big)-2n\sum_{xy}\sqrt{\mu_n(x)\mu_n(y)}\\ 
&\ge n\sum_{x\in \Xi_n}\mu_n(x)
\sum_{i=1}^d2\cosh\left(\frac{1}{2}\partial_i F(x)\right)-
2n\sum_{xy}\sqrt{\mu_n(x)\mu_n(y)} - C\ ,
\end{align*}
where $C\ge0$ accounts for an error arising when replacing
$\nabla^n_i F(x)$ and $\nabla^n_i F(x-\frac{e_i}{n})$ by
$\partial_i F(x)$.  By Young's inequality $2ab\le a^2+b^2$ and
$\sum_{x\in \Xi_n}\mu_n(x)=1$, we obtain
\begin{align*}
2n\sum_{xy}\sqrt{\mu_n(x)\mu_n(y)}\le 2dn\ .
\end{align*}
Hence, recalling the definition of $\ms G$ from \eqref{eq:G-cosh}, we
further get
\begin{align*}
n\, \ms I_n(\mu_n)&\ge 2dn+ n\sum_{x\in
\Xi_n}\mu_n(x)\sum_{i=1}^d2\left(\cosh\left(
\frac{1}{2}\partial_iF(x)\right)-1\right)- 2dn-C\\ 
&= n\sum_{x\in \Xi_n}\mu_n(x)\, \ms G(x) - C\ .
\end{align*}
Since
$\ms G(x)\ge \sum_{i=1}^d \left(\frac{1}{2}\partial_i F(x)\right)^2\ge
0$, the assumption in \eqref{eq:finite-energy} implies that
\begin{align*}
\limsup_{n\to \infty}n\sum_{x\in \Xi_n}\mu_n(x)\, |\nabla
F(x)|^2<\infty\ .
\end{align*}

	By assumptions (F1) and (F2), for every $\varepsilon\in (0,1)$ for which $\Xi_n\setminus \bigcup_{z\in \ms P} Q^n_\varepsilon(z)$ is non-empty, there exists $c=c(F)>0$ satisfying
	\begin{align*}
		\inf_{x\notin\bigcup_{z\in \ms P} Q^n_\varepsilon(z)} \ms G(x)\ge c>0\ .
	\end{align*}
	This proves the  claim in \eqref{eq:claim1}.  As for the  claim in \eqref{eq:claim2}, we observe that, by the non-degeneracy of $F$ around the critical points, i.e., (F3), for every $z\in \ms P$, there exists $b_z=b_z(F)>0$ such that, for all $\varepsilon \in (0,1)$ small enough so that $Q^n_\varepsilon(z)\cap Q^n_\varepsilon(z')=\varnothing$ for $z\neq z'\in \ms P$,	
	\begin{align*}
		|\nabla F(x)|^2\ge b_z\,|x-z|^2\ ,\qquad |x-z|<\varepsilon\ .
	\end{align*}
	This concludes the proof of the lemma.
\end{proof}

\end{lemma}
We readily obtain the following corollary of Lemma
\ref{lem:finite-energy}.

\begin{corollary}
\label{cor:finite-energy}
	If $\limsup_{n\to \infty}n\, \ms I_n(\mu_n)<\infty$, then
	\begin{align*}
		\lim_{\varepsilon\to 0}\limsup_{n\to \infty}n \sum_{z\in \ms P}\sum_{x\in Q^n_\varepsilon(z)}\mu_n(x)\,  |x-z|^3=0\ .
	\end{align*}
\end{corollary}
\begin{proof}
	For every $\varepsilon \in (0,1)$ and $n\in \bb N$, we have
	\begin{align*}
		n	\sum_{z\in \ms P}\sum_{x\in Q^n_\varepsilon(z)}\mu_n(x)\,  |x-z|^3\le C\, \varepsilon\, n	\sum_{z\in \ms P}\sum_{x\in Q^n_\varepsilon(z)}\mu_n(x)\, |x-z|^2\ ,
	\end{align*}
 for some $C=C(d)>0$.
The claim \eqref{eq:claim2} in Lemma \ref{lem:finite-energy} yields the desired result.
\end{proof}
We are finally ready to complete the proof of Proposition \ref{pr:step-1-RW}.
\begin{proof}[Proof of Proposition \ref{pr:step-1-RW}] We discuss the
proof of the \emph{$\Gamma$-liminf} and \emph{$\Gamma$-limsup}
inequalities separately.
	
\smallskip\noindent \emph{$\Gamma$-liminf}.
If $\mu\neq \sum_{z\in \ms P} \omega(z)\, \delta_z$, then $\ms J(\mu)>0$
and, by Lemma \ref{lem:zero-expansion}, the \emph{$\Gamma$-liminf}
inequality is automatically satisfied. Therefore, it remains to prove
the \emph{$\Gamma$-liminf} with
$\mu= \sum_{z\in \ms P}\omega(z)\, \delta_z$ and $\mu_n\to
\mu$. Moreover, since $\ms J^{(0)}(\mu)<\infty$ for such measures
$\mu$, without loss of generality we may further assume condition
\eqref{eq:finite-energy} for the sequence $\mu_n$.
	
Fix $\varepsilon\in (0,1)$ small enough for which
$Q_\varepsilon(z)\cap Q_\varepsilon(z')=\varnothing$ holds for all
$z\neq z'\in \ms P$. By \eqref{eq:claim1}, there exist
$C=C(\varepsilon)>0$ and a sequence
$\varepsilon_n \in (\varepsilon/2,\varepsilon)$ such that, for all
$n \in \bb N$ large enough,
\begin{align}
\label{eq:boundary-term}
n\,\mu_n\Big( \bigcup_{z\in \ms P}\partial
Q^n_{\varepsilon_n}(z)\Big)\le \frac{C}{n}\ .
\end{align}
\

Recall the definition of the restricted Dirichlet form $D_n(\ms A,h)$
from \eqref{eq:restricted-DF}, with $\ms A\subset \Xi$ and
$h:\Xi\to \bb R$. Hence, since $\{Q_{\varepsilon_n}(z)\}_{z\in \ms P}$
are disjoint subsets of $\Xi$, letting $f_n:=\mu_n/\pi_n$, we obtain
\begin{align*}
n\, \ms I_n(\mu_n)&= n\, D_n(\Xi_n,\sqrt{f_n}) \ge \sum_{z\in \ms P}
n\,\pi_n(Q_{\varepsilon_n}(z))\, D_n(Q_{\varepsilon_n}(z),\sqrt{f_n})\
.
\end{align*}
If $\color{blue}\mu_n^z:= \mu_n/\mu_n(Q_{\varepsilon_n}(z))$ denotes
the conditional measure $\mu_n$ on $Q^n_{\varepsilon_n}(z)$, and
$\color{blue}\ms I_n^z$ the level two large deviations rate functional
of the random walk reflected in $Q^n_{\varepsilon_n}(z)$ (which is
still reversible w.r.t.\ the conditional distribution
$\color{blue}\pi_n^z$ on $Q^n_{\varepsilon_n}(z)$), each term on the
right-hand side above further reads as follows:
\begin{align*}
& n\, \pi_n(Q_{\varepsilon_n}(z))\, D_n(Q_{\varepsilon_n}(z),\sqrt{f_n})\\
&= \frac{n}{2}\sum_{\substack{x,y \in Q^n_{\varepsilon_n}(z)\\
|x-y|=1/n}} e^{-\frac{n}{2}\left(F(x)+F(y)\right)}
\left(\,
e^{\frac{n}{2}F(x)}\sqrt{\mu_n(x)}-e^{\frac{n}{2}F(y)}\sqrt{\mu_n(y)}\,
\right)^2\\
&= \mu_n(Q_{\varepsilon_n}(z))\, n\, \ms I_n^z(\mu_n^z) \;,
\end{align*}
so that
\begin{align*}
n\, \ms I_n(\mu_n)\ge \sum_{z\in \ms P} 
\mu_n(Q_{\varepsilon_n}(z))\, n\, \ms I_n^z(\mu_n^z)\;.
\end{align*}
Thus, the \emph{$\Gamma$-liminf} inequality boils down to establishing
an estimate for each of the term on the right-hand side.  Fix
$z \in \ms P$ all throughout.

Letting $\color{blue}\ms L_n^z$ denote the infinitesimal generator of
a random walk reflected in $Q^n_{\varepsilon_n}(z)$, and observing
that $\ms L_n^z h(x) = \ms L_n h(x)$ for all
$x \in \mathring Q^n_{\varepsilon_n}(z)$ and
$h:Q^n_{\varepsilon_n}(z)\to \bb R$, the variational formulation of
$\ms I_n^z$ yields that
\begin{align}
\label{eq:nIn}
\begin{aligned}
& \mu_n(Q_{\varepsilon_n}(z))\, n\, \ms
I_n^z(\mu_n^z)\ge \sum_{x\in Q^n_{\varepsilon_n}(z)} \mu_n(x)\,
\frac{\big(-n\, \ms L_n^zu\big)(x)}{u(x)}\\
&\qquad= \sum_{x\in \mathring Q^n_{\varepsilon_n}(z)}\mu_n(x)\,
\frac{\big(-n\, \ms L_nu\big)(x)}{u(x)}+ \sum_{x\in \partial
Q^n_{\varepsilon_n}(z)}\mu_n(x)\, \frac{\big(-n\, \ms
L_n^zu\big)(x)}{u(x)}
\end{aligned}
\end{align}
for all $u: Q^n_{\varepsilon_n}(z)\to (0,\infty)$.

Consider the boundary terms, i.e., the second quantity on the
right-hand side of \eqref{eq:nIn}. Note that, for all
$G\in \ms C(\Xi)$ Lipschitz continuous,
\begin{align*}
\sup_{n\in \bb N}\sup_{x\in \Xi_n}\Big|\, \frac{\big(-\ms
L_n^ze^{nG}\big)(x)}{e^{nG(x)}} \,\Big|\le C 
\end{align*}
holds for some $C=C(F,G,d)>0$. Hence, by \eqref{eq:boundary-term}, we get
\begin{align}
\label{eq:bd-terms}
\Big|\,	\sum_{x\in \partial Q^n_{\varepsilon_n}(z)}
\mu_n(x)\,\frac{\big(-n\, \ms L_n^z e^{nG}\big)(x)}{e^{nG(x)}}
\,\Big|= O\Big(\frac{1}{n}\Big) \ . 
\end{align}

Consider now the first quantity on the right-hand side of
\eqref{eq:nIn}, with the idea of plugging a suitable $u=u_n=e^{n G}$
with $G\in \ms C^\infty(\bb R^d)$. For this purpose, consider the
$d\times d$ matrix obtained from ${\rm Hess}\, F(z)$ by keeping the same
eigenvectors, while taking only the negative of the negative part of
each eigenvalue: given the following matrix eigendecomposition
${\rm Hess}\, F(z)=U\Lambda U^{\sf T}$, with $U$ orthogonal and
$\Lambda$ diagonal matrices, define
$\color{blue} H:= U\Omega U^{\sf T} $, where $\color{blue}\Omega$ is a diagonal
matrix with diagonal entries given by
$\Omega_{ii} = -\, \max\{\:-\Lambda_{ii},0\:\}$.
Let 
\begin{equation}
\label{eq:definition-G}
{\color{blue}	G(x)}:=\frac{1}{2} (x-z)^{\sf T} H (x-z)\ .
\end{equation} 
A simple computation yields
\begin{align*}
&\frac{\big(-n\, \ms L_n e^{nG}\big)(x)}{e^{nG(x)}}\;
=\;  4n\sum_{i=1}^d  \sinh\left(\frac{1}{2}\partial_i
F(x)-\frac{1}{2}\partial_i G(x)\right)\sinh\left(\frac{1}{2}\partial_i
G(x)\right)\\ 
&\; -\; \frac{1}{2}\sum_{i=1}^d\left\{ \partial^2_{ii}F(x)\,
\cosh\left(\frac{1}{2}\partial_i F(x)\right)+
\partial^2_{ii}\left(2G-F\right)(x)\, \cosh\left(\frac{1}{2}\partial_i
(2G-F)(x)\right) \right\}\\ 
&\;  +\; O\left(\frac{1}{n}\right)
\; =: \; {\color{blue}n\,\ms R(x) + \ms S(x)} + O\left(\frac{1}{n}\right)\ .
\end{align*}
Therefore, 
\begin{align*}
&\liminf_{n\to \infty} \sum_{x\in \mathring
Q^n_{\varepsilon_n}(z)}\mu_n(x)\, \frac{\big(-n\, \ms L_n
e^{nG}\big)(x)}{e^{nG(x)}}\\ 
&\quad\ge \liminf_{n\to \infty}  \sum_{x\in \mathring
Q^n_{\varepsilon_n}(z)}\mu_n(x)\, n\,\ms R(x) +\liminf_{n\to \infty}
\sum_{x\in \mathring Q^n_{\varepsilon_n}(z)}\mu_n(x)\, \ms S(x)\ . 
\end{align*}

Since $\ms S\in \ms C(\Xi)$, $\ms S\ge 0$, $\varepsilon_n
>\varepsilon/2$,  and $\mu_n\to \mu= \sum_{z\in \ms P} \omega(z)\,
\delta_z$, we obtain 
\begin{equation*}
\liminf_{n\to \infty} \sum_{x\in \mathring
Q^n_{\varepsilon_n}(z)}\mu_n(x)\, \ms S(x) \ge \omega(z)\,\ms S(z) =
-\omega(z)\, {\rm Tr}(H)\ , 
\end{equation*}
where the last step follows by the definition of $G$ in
\eqref{eq:definition-G} (in particular, $\partial_i G(z)=0$ and
$\partial_{ii}G(z)=H_{ii}$) and the fact that $z\in \ms P$. Since
$-{\rm Tr}(H) = \zeta(z)$ (see \eqref{eq:zeta}), in view of the
definition \eqref{eq:I0} of $\ms J^{(0)}$, we are done as soon as we
show
\begin{align}
\label{final-liminf}
\lim_{\epsilon \to 0} \limsup_{n\to \infty}\sum_{x\in \mathring
Q^n_{\varepsilon_n}(z)} \mu_n(x)\,n\, \big|\, \ms R(x)\,\big|
\; =\; 0 \; . 
\end{align}

Since $z\in \ms P$ and $F\in \ms C^2(\Xi)$, there exists $C=C(F)>0$
such that
$\left|\partial_i F(x)\right| + \left|\partial_i G(x)\right|\le
C\,|x-z|$, for all $x\in Q^n_{\varepsilon_n}(z)$, $i=1,\ldots, d$.
Hence, since for any $b>0$ there exists $C(b)<\infty$ such that
$|\,\sinh (a) \, | \le C(b)$ and 
$|\,\sinh (a) - a\, | \le C(b) |a|^3$ for all $|a|\le b$,
\begin{equation*}
\big|\, \ms R(x)\,\big| \;\le\;
\Big|\, \sum_{i=1}^d \big[\, \partial_i F(x)
-\partial_i G(x)\big]\, \partial_i G(x) \,\Big| \;+\; C_0 \, |x-z|^3
\end{equation*}
for some finite constant $C_0$ whose value may change from line to
line. By condition (F1), $F$ belongs to $\ms C^{2,1}$. Hence, since
$z$ is a critical point of $F$, by a Taylor expansion
\begin{equation*}
\Big|\, \partial_i F(x) - \sum_{j=1}^d
\partial^2_{ij}F(z)\, (x_j-z_j)\, \Big|\;\le\; C_0 |x-z|^2\;.
\end{equation*}
A similar expansion holds for $G$. Hence, as
$|\partial_i F(x)| + \left|\partial_i G(x)\right|\le C\,|x-z|$, and
$\partial^2_{ij}G(z) = H_{ij}$,
\begin{equation*}
\big|\, \ms R(x)\,\big| \;\le\;
\Big|\, \sum_{i=1}^d \sum_{j=1}^d
\big[\, \partial^2_{ij}F(z) - H_{ij} \, \big] (x_j-z_j)
\sum_{k=1}^d H_{ik}  (x_k-z_k)
\,\Big| \;+\; C_0 \, |x-z|^3\;.
\end{equation*}
Since $H$ is symmetric, we may rewrite the previous sum as
\begin{align*}
& \sum_{j,k=1}^d (x_j-z_j)(x_k-z_k)\, \Big(\sum_{i=1}^d
\big\{ \, [ {\rm Hess}\, F(z)\,]_{ij} \,-\, H_{ij}\,\big\} \,
H_{ki}\, \Big) \\
&\quad=(x-z)^{\sf T}\left[\, {\rm Hess}\, F(z)-H\right]\,  H \:(x-z)\;.
\end{align*}
This expression vanishes because, by definition of $H$,
${\rm Hess}\, F(z) \,H \,=\, H^2$. This proves that
$|\, \ms R(x)\, | \;\le\; C_0 \, |x-z|^3$. Hence, \eqref{final-liminf}
follows from Corollary \ref{cor:finite-energy}.

\smallskip\noindent \emph{$\Gamma$-limsup.} 
 We must prove the
inequality for all $\mu\in \ms P(\Xi)$ such that
$\ms J^{(0)}(\mu)<\infty$, i.e., of the form
$\mu=\sum_{z\in \ms P}\omega(z)\, \delta_z$. Furthermore, since
$\ms I_n$ is convex and $\ms J^{(0)}$ is linear on its support, it
suffices to consider $\mu=\delta_z$, $z\in \ms P$.

Fix a critical point $z \in \ms P$ and, for notational convenience, assume $z=0$. Recalling the matrix eigendecomposition ${\rm Hess}_F(0)=U\Lambda U^{\sf T}$ from the proof of the \emph{$\Gamma$-liminf} inequality above, define the following $d\times d$ matrix $\color{blue}W:=U\Sigma U^{\sf T}$, where $\color{blue}\Sigma$ is diagonal and such that $\Sigma_{ii}:= \max\{\:-\Lambda_{ii},0\:\}$.
Let 
\begin{equation}
	{\color{blue}K(x)}:= - x^{\sf T} W x\ ,
\end{equation}
and note that ${\rm Hess}_F(0)+2W\in \bb R^{d\times d}$ is positive-definite by (F3), with eigenvalues given by the absolute value of those of ${\rm Hess}_F(0)$.
Next, we introduce:
\begin{itemize}
	\item a sequence $\color{blue}\varepsilon_n$ satisfying $n^{1/3}\ll \varepsilon_n^{-1}\ll n^{1/2}$;
	\item  a smooth cutoff function $\color{blue}\varphi:\Xi\to [0,1]$ with compact support strictly contained in $Q_1(0)$, satisfying the following two conditions: for all functions $\ms B\in \ms C(\Xi)$ and for all positive-definite $d\times d$ matrices $T$, letting ${\color{blue}\varphi_n(x)}:=\varphi(\varepsilon_nx)$, 
	\begin{equation}\label{eq:cutoff1}
		\lim_{n\to \infty} \frac{\sum_{x\in \Xi_n}e^{-n\, x^{\sf T}T x}\, \big(\, \varphi_n(x)\,\big)^2\,\ms B(x)}{\sum_{x\in \Xi_n}e^{-n\, x^{\sf T}T x}\, \big(\, \varphi_n(x)\,\big)^2}= \ms B(0)\ ,
	\end{equation}
	\begin{equation}\label{eq:cutoff2}
		\lim_{n\to \infty}	\frac{\sum_{x\in \Xi_n}e^{-n\, x^{\sf T}Tx}\,n\, \big(\, \varphi_n(x+\tfrac{e_i}{n})-\varphi_n(x)\,\big)^2}{\sum_{x\in \Xi_n} e^{-n\, x^{\sf T}Tx}\big(\, \varphi_n(x)\,\big)^2}=0\ .
	\end{equation} 
\end{itemize}

We now construct a recovery sequence $\mu_n$ for $\mu=\delta_0$.  Let $g_n:\Xi_n\to \bb R$ be given by
\begin{equation}
	{\color{blue}g_n(x)}:= e^{\frac{n}{2}K(x)}\varphi_n(x) \ ,
\end{equation}
and define ${\color{blue}\mu_n}:= g_n^2\, \pi_n/A_n$, where ${\color{blue}A_n}:= \sum_{y\in \Xi_n} g_n(y)^2\, \pi_n(y)$ denotes the normalization constant.

As a consequence of these definitions, we have that $\mu_n\to \mu=\delta_0$ in $\ms P(\Xi)$. Indeed, for all $\ms B\in \ms C(\Xi)$, we have
\begin{align*}
	&\sum_{x\in \Xi_n}\mu_n(x)\, \ms B(x) =\frac{\sum_{x\in \Xi_n}e^{-n \left(F(x)-F(0)\right)+ n K(x)}\big(\,\varphi_n(x)\,\big)^2\, \ms B(x)}{\sum_{x\in \Xi_n}e^{-n \left(F(x)-F(0)\right)+ n K(x)}\big(\,\varphi_n(x)\,\big)^2}\\
	&\qquad=\frac{\sum_{x\in Q_n}e^{- n L(x)}\big(\,\varphi_n(x)\,\big)^2\, \ms B(x)\left(1+o(\varepsilon_n^3 n)\right)}{\sum_{x\in Q_n}e^{- n L(x)}\big(\,\varphi_n(x)\,\big)^2\left(1+o(\varepsilon_n^3n)\right)}\ ,
\end{align*}
where we abbreviated ${\color{blue}L(x)}:= \frac{1}{2}x^{\sf T}{\rm Hess}_F(0) x - K(x)$ and ${\color{blue} Q_n}:= Q^n_{\varepsilon_n}(0)$.
Since  $L(x)$ is strictly quadratic, by \eqref{eq:cutoff1}, this expression converges, as $n\to \infty$, to $\ms B(0)$, thus, proving the claim.

We now prove $\lim_{n\to \infty}n\, \ms I_n(\mu_n)=\ms J^{(0)}(\mu)$, with $\mu=\delta_0$. After a simple manipulation, $n\, \ms I_n(\mu_n)$ reads as follows:
\begin{align*}
	n\, \ms I_n(\mu_n)&= \frac{\sum_{x \in \Xi_n}\sum_{i=1}^d \pi_n(x)\, R_n(x,x+\tfrac{e_i}{n})\, n\, \big(\, g_n(x+\tfrac{e_i}{n})-g_n(x)\,\big)^2}{A_n}\\
	&\qquad=: \frac{\sum_{x\in Q_n}e^{-n\left(F(x)-F(z)- K(x)\right)}\, \Phi_n(x)}{\sum_{x\in Q_n}e^{-n\left( F(x)-F(z)- K(x)\right)}\,\big(\,\varphi_n(x)\,\big)^2}\ .
\end{align*}
In this formula,  ${\color{blue}\Phi_n(x)}$ denotes
\begin{equation*}
	\sum_{i=1}^d e^{-\frac{1}{2}\nabla^n_i F(x)+ \nabla^n_i K(x)}\,n\,\big[\, \varphi_n(x)\big(\,1-e^{-\frac{1}{2}\nabla^n_iK(x)}\,\big)+\big(\,\varphi_n(x+\tfrac{e_i}{n})-\varphi_n(x)\,\big) \,\big]^2\ ,
\end{equation*}
where we recall   $\nabla^n_i F(x)=n\, \big(\, F(x+\tfrac{e_i}{n})-F(x)\, \big)$; similarly for $\nabla^n_i K(x)$. Recall that $z=0$ is a critical point for both $F$ and $K$. Moreover, since both functions are twice differentiable in a neighbourhood of $z=0$, we get
\begin{equation*}
	\limsup_{n\to \infty}\max_{i=1,\ldots, d}\sup_{x \in Q_n} e^{-\frac{1}{2}\nabla^n_i F(x)+ \nabla^n_i K(x)}\le 1\ .
\end{equation*}
Therefore,
\begin{equation}\label{eq:limsup}
	\limsup_{n\to \infty} n\, \ms I_n(\mu_n)\le \limsup_{n\to \infty} \frac{
		\sum_{x\in Q_n}e^{- nL(x)}\, \widetilde \Phi_n(x)\left(1+o(\varepsilon_n^3n)\right)
	}{
		\sum_{x\in Q_n}e^{- nL(x)}\big(\,\varphi_n(x)\,\big)^2\left(1+o(\varepsilon_n^3n)\right)
	}\ ,
\end{equation}
where
\begin{equation}\label{eq:square}
	{\color{blue}\widetilde \Phi_n(x)}:= \sum_{i=1}^d n\,\big[\, \varphi_n(x)\big(\,1-e^{-\frac{1}{2}\nabla^n_iK(x)}\,\big)+\big(\,\varphi_n(x+\tfrac{e_i}{n})-\varphi_n(x)\,\big) \,\big]^2\ .	
\end{equation}
By expanding the square  in \eqref{eq:square}, we obtain three terms:
\begin{align*}
	\Psi_{n,1}(x)&:=	\sum_{i=1}^d n\,\big[\, \varphi_n(x)\big(\,1-e^{-\frac{1}{2}\nabla^n_iK(x)}\,\big)\big]^2
	\\
	\Psi_{n,2}(x)&:=\sum_{i=1}^d n\big[\,\varphi_n(x+\tfrac{e_i}{n})-\varphi_n(x)\,\big]^2\\
	\Psi_{n,3}(x)&:= 2\sum_{i=1}^d n\big[\, \varphi_n(x)\big(\,1-e^{-\frac{1}{2}\nabla^n_iK(x)}\,\big)\big]\big[\,\varphi_n(x+\tfrac{e_i}{n})-\varphi_n(x)\,\big]\ ,
\end{align*}
 yielding a sum of three expressions on the right-hand side of \eqref{eq:limsup}. Let us consider the limit of the first one of these expressions:
\begin{align*}
	& \frac{
		\sum_{x\in Q_n}e^{- nL(x)}\, \Psi_{n,1}(x)\left(1+o(\varepsilon_n^3n)\right)
	}{
		\sum_{x\in Q_n}e^{- nL(x)}\big(\,\varphi_n(x)\,\big)^2\left(1+o(\varepsilon_n^3n)\right)
	}\\
	&\quad=
	\frac{
		\sum_{x\in Q_n}e^{- nL(x)}\, \big(\,\varphi_n(x)\,\big)^2 \sum_{i=1}^d n \big[\,1-e^{-\frac{1}{2}\nabla^n_i K(x)}\,\big]^2\left(1+o(\varepsilon_n^3n)\right)
	}{
		\sum_{x\in Q_n}e^{- nL(x)}\big(\,\varphi_n(x)\,\big)^2\left(1+o(\varepsilon_n^3n)\right)
	}\ ,
\end{align*}
and show that it converges to the function $\zeta$ given in \eqref{eq:zeta} evaluated at $z=0$. 
This would conclude the proof of the \emph{$\Gamma$-limsup} inequality because the second one of these expressions vanishes as $n\to \infty$ by \eqref{eq:cutoff2}, thus, the third one as well by  H\"older's inequality.

Recall  $L(x)=\frac{1}{2} x^{\sf T}\left({\rm Hess}_F(0)+2W\right)x$. Since ${\rm Hess}_F(0)+2W$ is positive-definite and $\varepsilon_n^{-1}=o(n^{1/2})$, \eqref{eq:cutoff1} ensures
\begin{align*}
	&	\lim_{n\to \infty} \frac{
		\sum_{x\in Q_n}e^{- nL(x)}\, \big(\,\varphi_n(x)\,\big)^2 \sum_{i=1}^d n \big[\,1-e^{-\frac{1}{2}\nabla^n_i K(x)}\,\big]^2\left(1+o(\varepsilon_n^3n)\right)
	}{
		\sum_{x\in Q_n}e^{- nL(x)}\big(\,\varphi_n(x)\,\big)^2\left(1+o(\varepsilon_n^3n)\right)
	}\\
	&=\lim_{n\to \infty} \frac{
		\sum_{x\in Q_n}e^{- nL(x)} \sum_{i=1}^d n \big[\,\frac{1}{2}\partial_i K(x)\,\big]^2
	}{
		\sum_{x\in Q_n}e^{- nL(x)}
	}= \lim_{n\to \infty} \frac{\sum_{x\in Q_n} e^{-nL(x)}\, n\, x^{\sf T}W^2x}{\sum_{x\in Q_n} e^{-nL(x)}}\ .
\end{align*}
By rewriting this last fraction as an approximation  over the grid $\bb Z^d/\sqrt{n}$ of a Gaussian integral, the above limit  coincides with $\sum_{i=1}^d \frac{1}{|\xi_i(0)|}(\max\{\:-\xi_i(0),0\:\})^2$, i.e., the function $\zeta$ given in \eqref{eq:zeta} evaluated at $z=0$. This proves the \emph{$\Gamma$-limsup} inequality and, thus, concludes the proof of the proposition.
\end{proof}

\subsection{$\Gamma$-convergence of $n\, e^{d_pn}\, \ms I_n$}
\label{sec6.2}

We turn to the last terms of the expansion. The proof is based on
Theorem \ref{mt1}. The first condition of this theorem assumes that
the rate functional $\ms I_n$ $\Gamma$-converges to a functional whose
$0$-level set consists in the convex combinations of a finite number
of Dirac measures. In the present context of a random walk in a
potential field, this condition is not satisfied by $\ms I_n$, but by
$n\, \ms I_n$. To adjust the model to the hypotheses of the theorem we
need to change the time-scale.

Let $Y^{(n)}_t$ be the $\Xi_n$-valued Markov chain defined by
$\color{blue} Y^{(n)}_t = X^{(n)}_{tn}$. As we speeded-up the chain by
$n$, the level two large deviations rate functional of the chain
$Y^{(n)}_t$, denoted by $\color{blue} \ms K_n$, is simply given by
$n\ms I_n$: $\ms K_n = n\, \ms I_n$. By Proposition \ref{pr:step-1-RW},
condition (H0), with $\ms J^{(0)}$ playing the role of $\ms I^{(0)}$,
is satisfied by the sequence $\ms K_n$. Recall the definition of the
functional $\ms J^{(p)}$ introduced just above the statement of
Theorem \ref{mt3}.

\begin{lemma}
\label{lf1}
Assume that conditions {\rm (F1) -- (F6)} are in force. Then,
conditions {\rm (H0) -- (H5)} are fulfilled by the Markov chain
$Y^{(n)}_t$ and its associated large deviations rate functional
$\ms K_n$.
\end{lemma}

\begin{proof}
As observed above, by Proposition \ref{pr:step-1-RW}, condition (H0)
is fulfilled. We turn to conditions (H1) -- (H5).  By \cite[Lemma
6.3]{lmt2015}, (the definition of $\beta_m$ appearing in the statement
of this lemma is given just above equation (2.8)) and equations (2.9)
and (2.4) in \cite{lmt2015}, condition (H1) of Section \ref{sec1} is
fulfilled for $\theta^{(p)}_n = e^{d_pn}$. Note that  in comparison to \cite{lmt2015}, there is  a
factor $n$ less here because the
process $Y^{(n)}_t$ has already been speeded-up by $n$.

We turn to Condition (H2). The tree structure of the wells for the
dynamics considered here is quite simple. For each $1\le p\le \mf q$,
the well $\ms V^{p,j}_n$ is a connected component of the set
$\{x\in \Xi: F(x) < h\} \cap \Xi_n$. Morevoer, by \eqref{v23}, for
each $j\in S_p$, either $\ms V^{p,j}_n = \ms V^{p+1,k}_n$ for some
$k\in S_{p+1}$ or $\ms V^{p,j}_n \subset \Delta_{p+1}$.  Condition
(H3) holds in view of the definition \eqref{v23} of the jump rates
$\bs r_{\mf q}$ and from the fact that $\mc M_{\mf q+1}=\varnothing$.

Condition (H4a) holds because a neighbourhood of each global minimum
of $F$ is contained in a valley. Condition (H4b) is satisfied by
construction. Finally, condition (H5) follows from Lemma \ref{l04}.
\end{proof}

Next result is a consequence of Theorem \ref{mt1} and the previous lemma.

\begin{corollary}
\label{lf2}
Assume that conditions {\rm (F1) -- (F6)} are in force. Then, for each
$1\le p \le \mf q$, $e^{d_pn} \, \ms K_n$ $\Gamma$-converges to
$\ms J^{(p)}$. Moreover, for each $0\le p<\mf q$, the $0$-level set of
$\ms J^{(p)}$ corresponds to the set where $\ms J^{(p+1)}$ is finite,
and the $0$-level set of $\ms J^{(\mf q)}$ is a singleton.
\end{corollary}

\subsection{Proof of Theorem \ref{mt3}}

Condition (a) of Definition \ref{def1} is clearly satisfied. To prove
condition (b), notice that, by Lemma \ref{lem:zero-expansion} and
Proposition \ref{pr:step-1-RW}, $\ms I_n$, $n \ms I_n$
$\Gamma$-converge to $\ms J$, $\ms J^{(0)}$, respectively. By Lemma
\ref{lf1}, $e^{d_pn} \, \ms K_n$ $\Gamma$-converges to $\ms J^{(p)}$,
$1\le p\le \mf q$. Hence, since $\ms K_n = n \ms I_n$,
$n\, e^{d_pn} \, \ms I_n$ $\Gamma$-converges to $\ms J^{(p)}$.

We turn to condition (c). By definition of $\ms J$ and $\ms G$, the
$0$-level set of $\ms J$ consists of convex combinations of Dirac
measures supported at elements of $\mc P$, the set of critical points
of $F$. By \eqref{eq:I0}, \eqref{eq:zeta}, this set corresponds to the
set where $\ms J^{(0)}$ is finite. By Corollary \ref{lf2}, condition
(c) holds for $0\le p<\mf q$ as well as condition (d). This completes
the proof. \qed

\subsection*{Acknowledgments}

C.~L.\ has been partially supported by FAPERJ CNE E-26/201.207/2014, by
CNPq Bolsa de Produtividade em Pesquisa PQ 303538/2014-7.  R.~M.\ has
been supported by CNPq, grant Universal no.\ 403037/2021-2. F.~S.\ wishes
to thank IMPA for the very kind and warm hospitality during his stay,
as well as for partial support. The same author acknowledges partial
support from the Lise Meitner fellowship, Austrian Science Fund (FWF):M3211.

\end{document}